\documentclass[oneside,english]{amsart}
\usepackage[latin9]{inputenc}
\usepackage{amsthm}
\usepackage{amssymb}
\usepackage{stmaryrd}

\makeatletter
\numberwithin{equation}{section}
\numberwithin{figure}{section}

\usepackage{amsmath}

\usepackage{amssymb}
\usepackage{stmaryrd}
\usepackage{amsthm}
\usepackage[mathscr]{eucal}
\usepackage{amscd}
\usepackage[all]{xy}
\usepackage{url}

\newtheorem{Thm}{Theorem}[section]
\usepackage{aliascnt}

\usepackage{enumerate}

\usepackage{comment}

\usepackage{hyperref}
\hypersetup{%
pdftitle={Preprint},
pdfauthor={},
pdfkeywords={},
bookmarksnumbered,
pdfstartview={FitH},
breaklinks=true,
colorlinks=true,
urlcolor=blue,
citecolor=blue,
}%
\usepackage[all]{hypcap} 

\usepackage{breakurl}

\usepackage{color}

\theoremstyle{plain}
\newtheorem{Lem}[Thm]{Lemma}
\theoremstyle{plain}
\newtheorem{Prop}[Thm]{Proposition}
\theoremstyle{plain}
\newtheorem{Cor}[Thm]{Corollary}
\theoremstyle{plain}

\theoremstyle{plain}
\newtheorem*{Thm*}{Theorem}
\theoremstyle{plain}
\newtheorem*{Conj*}{Conjecture}

\theoremstyle{definition}
\newtheorem{Def}[Thm]{Definition}
\theoremstyle{definition}
\newtheorem*{Def*}{Definition}
\theoremstyle{definition}
\newtheorem{Eg}[Thm]{Example}
\theoremstyle{definition}

\theoremstyle{definition}
\newtheorem{Assumption}{Assumption}
\theoremstyle{definition}

\theoremstyle{definition}
\newtheorem{Property}{Property}
\theoremstyle{definition}

\theoremstyle{definition}

\theoremstyle{remark}
\newtheorem{Rem}[Thm]{Remark}

\newcommand{\kk}{\Bbbk}
\newcommand{\Z}{\mathbb{Z}}
\newcommand{\N}{\mathbb{N}}
\newcommand{\Q}{\mathbb{Q}}

\newcommand{\R}{\mathbb{R}}

\renewcommand{\hat}[1]{\widehat{#1}}
\renewcommand{\tilde}[1]{\widetilde{#1}}

\newcommand{\opname}[1]{\operatorname{\mathsf{#1}}}

\newcommand{\pr}{\opname{pr}}

\newcommand{\inj}{\opname{inj}}

\newcommand{\ra}{\rightarrow}

\newcommand{\codim}{\opname{codim}}

\newcommand{\sign}{\opname{sign}}

\newcommand{\Aut}{\opname{Aut}}

\newcommand{\Hom}{\opname{Hom}}

\newcommand{\ext}{\opname{ext}}

\newcommand{\supp}{\opname{supp}}

\renewcommand{\deg}{\opname{deg}}

\newcommand{\Hf}{{\frac{1}{2}}}

\newcommand{\Rm}[1]{{\longmapsto}}
\newcommand{\Lm}[1]{{\longmapsfrom}}

\newcommand{\cA}{{\mathcal A}}

\newcommand{\cC}{{\mathcal C}}

\newcommand{\cF}{{\mathcal F}}

\newcommand{\cS}{{\mathcal S}}

\newcommand{\cZ}{{\mathcal Z}}

\newcommand{\bL}{{\mathbf L}}

\newcommand{\sQ}{{\mathbb Q}}

\newcommand{\sT}{{\mathbb T}}

\newcommand{\ui}{{\underline{i}}}

\newcommand{\uk}{{\underline{k}}}

\newcommand{\tB}{{\widetilde{B}}}

\newcommand{\tE}{{\widetilde{E}}}
\newcommand{\tF}{{\widetilde{F}}}

\newcommand{\hA}{{\widehat{A}}}
\newcommand{\hB}{{\widehat{B}}}

\newcommand{\hG}{{\widehat{G}}}

\newcommand{\can}{L}
\newcommand{\clAlg}{{\cA}}
\newcommand{\qClAlg}{\cA_q}

\newcommand{\diag}{{\delta}}

\newcommand{\res}{{\mathrm{Res}}}

\usepackage{tikz}
\usetikzlibrary{positioning,shapes,shadows,arrows,snakes}
\usetikzlibrary {arrows.meta,positioning}

\pgfdeclarelayer{edgelayer}
\pgfdeclarelayer{nodelayer}
\pgfsetlayers{edgelayer,nodelayer,main}

\tikzstyle{none}=[inner sep=0pt]
\tikzstyle{black box}=[draw=black, fill=black!25]
\tikzstyle{white box}=[draw=black, fill=white]
\tikzstyle{black circle}=[circle,draw=black!50, fill=black!25]
\tikzstyle{red circle}=[circle,draw=red!50, fill=red!25]
\tikzstyle{blue circle}=[circle,draw=blue!50, fill=blue!25]
\tikzstyle{green circle}=[circle,draw=green!50, fill=green!25]
\tikzstyle{yellow circle}=[circle,draw=yellow!50, fill=yellow!25]

\newcommand{\thistheoremname}{}
\newtheorem*{genericthm*}{\thistheoremname}
\newenvironment{namedthm*}[1]
  {\renewcommand{\thistheoremname}{#1}%
   \begin{genericthm*}}
  {\end{genericthm*}}

\renewcommand{\diag}{{d}}

\renewcommand{\inj}{{\bI}}
\renewcommand{\can}{{\bL}}

\newcommand{\fv}{\opname{f}}
\newcommand{\ufv}{\opname{uf}}

\newcommand{\suppDim}{\opname{suppDim}}
\newcommand{\ptSet}{\mathcal{PT}}

\makeatother

\usepackage{babel}
\begin{document}
\renewcommand{\res}{\opname{res}}
\newcommand{\frR}{\mathfrak{R}}
\newcommand{\frI}{\mathfrak{I}}
\newcommand{\frz}{\mathfrak{f}}
\renewcommand{\diag}{\mathbf{d}}

\newcommand{\qO}{\mathcal{O}_q}

\newcommand{\up}{\opname{up}}
\newcommand{\dCan}{\opname{B^{\up}}}
\newcommand{\cdCan}{\opname{\mathring{B}^{\up}}}

\renewcommand{\qClAlg}{{\clAlg_q}}

\newcommand{\cone}{M^\circ}
\newcommand{\uCone}{\underline{\cone}}
\newcommand{\yCone}{N_{\ufv}}

\newcommand{\sol}{\mathrm{TI}}
\newcommand{\intv}{{\mathrm{BI}}}

\newcommand{\Perm}{\mathrm{P}}

\newcommand{\tui}{\widetilde{\ui}}
\newcommand{\tuk}{\widetilde{\uk}}

\newcommand{\bideg}{\opname{bideg}}

\newcommand{\ubeta}{\underline{\beta}}
\newcommand{\udelta}{\underline{\delta}}

\newcommand{\LP}{{\mathcal{LP}}}
\newcommand{\hLP}{{\widehat{\mathcal{LP}}}}

\newcommand{\bLP}{{\overline{\mathcal{LP}}}}
\newcommand{\bClAlg}{{\overline{\clAlg}}}
\newcommand{\bUpClAlg}{{\overline{\upClAlg}}}

\newcommand{\midClAlg}{{\clAlg^{\mathrm{mid}}}}
\newcommand{\upClAlg}{\mathcal{U}}

\newcommand{\bQClAlg}{{\bClAlg_q}}

\newcommand{\qUpClAlg}{{\upClAlg_q}}

\newcommand{\bQUpClAlg}{{\bUpClAlg_q}}

\newcommand{\canClAlg}{{\clAlg^{\mathrm{can}}}}

\newcommand{\AVar}{\mathbb{A}}
\newcommand{\XVar}{\mathbb{X}}

\newcommand{\Jac}{\hat{\mathop{J}}}

\newcommand{\wt}{\mathrm{cl}}
\newcommand{\cl}{\mathrm{cl}}

\newcommand{\domCone}{{\overline{M}^\circ}}

\newcommand{\sd}{\mathbf t}
\newcommand{\ssd}{\mathbf s}

\newcommand{\tree}{{\mathbb{T}}}

\newcommand{\img}{{\mathrm{Im}}}

\newcommand{\Id}{{\mathrm{Id}}}
\newcommand{\prin}{{\mathrm{prin}}}

\newcommand{\mm}{{\mathbf{m}}}

\newcommand{\frn}{{\mathfrak{n}}}
\newcommand{\frsl}{{\mathfrak{sl}}}

\newcommand{\frg}{\mathfrak{g}}
\newcommand{\hfrg}{\hat{\mathfrak{g}}}
\newcommand{\hfrh}{\hat{\mathfrak{h}}}

\newcommand{\frh}{\mathfrak{h}}
\newcommand{\frp}{\mathfrak{p}}
\newcommand{\frd}{\mathfrak{d}}
\newcommand{\frj}{\mathfrak{j}}
\newcommand{\frD}{\mathfrak{D}}
\newcommand{\frS}{\mathfrak{S}}
\newcommand{\frC}{\mathfrak{C}}
\newcommand{\frM}{\mathfrak{M}}

\newcommand{\Int}{\mathrm{Int}}
\newcommand{\ess}{\mathrm{ess}}
\newcommand{\Mono}{\mathrm{Mono}}

\newcommand{\bfm}{{\mathbf{m}}}
\newcommand{\bfI}{{\mathbf{I}}}

\newcommand{\Pot}{\mathrm{Pot}}
\newcommand{\kGp}{\opname{K}}

\newcommand{\s}{\mathrm{s}}
\newcommand{\fd}{\mathrm{fd}}

\renewcommand{\sc}{\mathrm{sc}}
\newcommand{\Hall}{\mathrm{Hall}}
\newcommand{\income}{\mathrm{in}}

\newcommand{\tn}{\tilde{n}}

\newcommand{\sing}{\opname{sing}}

\renewcommand{\inj}{\opname{inj}}

\renewcommand{\ext}{\mathrm{ext}}

\newcommand{\cJac}{\Jac}

\newcommand{\stilt}{\mathrm{s}\tau\mathrm{-tilt}}
\newcommand{\Fac}{\mathrm{Fac}}
\newcommand{\Sub}{\mathrm{Sub}}

\newcommand{\rigid}{\mathrm{rigid}}
\newcommand{\tauRigid}{\tau\mathrm{-rigid}}
\newcommand{\spTilt}{\mathrm{s}\tau\mathrm{-tilt}}
\newcommand{\clTilt}{\mathrm{c-tilt}}
\newcommand{\maxRigid}{\mathrm{m-rigid}}

\newcommand{\Li}{\mathrm{Li}}

\newcommand{\Trop}{\opname{Trop}}

\newcommand{\seq}{\boldsymbol{\mu}}

\newcommand{\val}{\mathbf{v}}
\newcommand{\hookuparrow}{\mathrel{\rotatebox[origin=c]{90}{$\hookrightarrow$}}} 
\newcommand{\hookdownarrow}{\mathrel{\rotatebox[origin=c]{-90}{$\hookrightarrow$}}}
\newcommand{\twoheaddownarrow}{\mathrel{\rotatebox[origin=c]{-90}{$\twoheadrightarrow$}}}

\newcommand{\rd}{{\opname{red}}} 
\newcommand{\bCan}{\overline{\can}} 

\title[]{Applications of the freezing operators on cluster algebras}
\dedicatory{Dedicated to Professor Jie Xiao on the occasion of his sixtieth birthday}
\author{Fan Qin}
\email{qin.fan.math@gmail.com}
\begin{abstract}
We apply freezing operators to relate different (quantum) upper cluster
algebras. We prove that these operators send localized (quantum) cluster
monomials to localized (quantum) cluster monomials. They also send
bases to bases in many cases. In addition, the bases constructed via
freezing coincide with those constructed via localization.
\end{abstract}

\maketitle
\tableofcontents{}

\section{Introduction \label{sec:intro} }

Cluster algebras were invented by Fomin and Zelevinsky \cite{fomin2002cluster}
to study the total positivity \cite{Lusztig96} and the dual canonical
bases of quantum groups \cite{Lusztig90,Lusztig91}\cite{Kas:crystal}.
Their quantization was introduced in \cite{BerensteinZelevinsky05}.
It turns out the (quantized) coordinate rings of many interesting
varieties are (quantum) upper cluster algebras, see \cite{BerensteinFominZelevinsky05}
\cite{GeissLeclercSchroeer10,GeissLeclercSchroeer11} \cite{GY13,goodearl2016berenstein,goodearl2020integral}
\cite{elek2021bott} \cite{shen2021cluster} \cite{galashin2022braid}
\cite{casals2022cluster} or \cite{qin2023analogs} for a continuously
expanding list. These varieties are often closely related. For example,
some can be obtained from others by taking locally closed subsets.
In \cite{qin2023analogs}, the author constructed analogs of the dual
canonical bases for these coordinate rings with the help of various
operations in the theory of cluster algebra (cluster theory for short).
In particular, the freezing operators were introduced to related different
cluster algebras.

In this paper, we focus on the study of freezing operators and their
applications to cluster algebras.

\subsection*{Freezing cluster monomials}

In the classical case, the upper cluster algebra is the coordinate
ring of a cluster variety, which is endowed with many toric charts
(called seeds in cluster theory). A localized cluster monomial is
an element of the upper cluster algebra such that it is a Laurent
monomial in some chart (\cite[Lemma 7.8]{gross2018canonical})

Take any initial (quantum) seed $t_{0}$. Let $I$ denote its set
of vertices, which are endowed with a partition $I=I_{\ufv}\sqcup I_{\fv}$
into unfrozen ones and frozen ones. We choose a set of unfrozen vertices
$F\subset I_{\ufv}$. By freezing the vertices in $F$, we obtain
a new seed $\frz_{F}t_{0}$ from the seed $t_{0}$. Note that $t_{0}$
determines the upper cluster algebra $\upClAlg(t_{0})$ and $\frz_{F}t_{0}$
determines a different upper cluster algebra $\upClAlg(\frz_{F}t_{0})$.

By \cite{FominZelevinsky07}\cite{Tran09}\cite{DerksenWeymanZelevinsky09}\cite{gross2018canonical},
the Laurent expansion of any localized (quantum) cluster monomial
$z$ in seed $t_{0}$ takes the form $z=x^{g}\cdot F_{z}|_{y_{k}\mapsto x^{p^{*}e_{k}}}$,
where $g$ denotes the (extended) $g$-vector, $F_{z}$ is a polynomial
in variables $y_{k}$, $k\in I_{\ufv}$, called the (quantum) $F$-polynomial,
and we substitute $y_{k}$ by a Laurent monomial $x^{p^{*}e_{k}}$
in the above formula. 

The freezing operator sends $z$ to $z':=x^{g}\cdot F_{z}'|_{y_{k}\mapsto x^{p^{*}e_{k}}}$,
where $F_{z}':=F_{z}|_{y_{j}\mapsto0,\ \forall j\in F}$. Our first
main result is the following.\begin{Thm}[{Theorem \ref{thm:projection_quantum_cluster_monomial}}]

\label{thm:intro_project_cluster_monomial}The freezing operator sends
any localized (quantum) cluster monomial of $\upClAlg(t_{0})$ to
a localized (quantum) cluster monomial of $\upClAlg(\frz_{F}t_{0})$.

\end{Thm}

\begin{Rem}

Our proof of Theorem \ref{thm:intro_project_cluster_monomial} heavily
relies on the classical scattering diagrams \cite{gross2018canonical}. 

Scattering diagrams behave well under freezing. For example, \cite{muller2015existence}
observed that freezing of seeds induces inclusion of chambers. His
observation has important consequences on mutation reachability, see
Lemma \ref{lem:projection_reachable_chamber}, Theorem \ref{thm:freeze_inj_reachable}(\cite[Theorem 1.4.1]{muller2015existence}),
Theorem \ref{thm:seed_rechability_reduction} (also proved in \cite{cao2020enough}
based on Muller's observation).

It would be desirable to know if there exists a different approach
to Theorem \ref{thm:intro_project_cluster_monomial} and results concerning
mutation reachability.

\end{Rem}

\subsection*{Constructions of bases}

Now, assume that the seed $t_{0}$ can be quantized or, equivalently,
the associated $B$-matrix $\tB(t_{0})$ is of full rank. There have
been many works towards the construction of well-behaved bases for
classical or quantum upper cluster algebras $\upClAlg(t_{0})$. Notably,
three important families of bases have been constructed in literature: 
\begin{itemize}
\item the generic bases in the sense of \cite{dupont2011generic}, see also
\cite{ding2009integral} \cite{GeissLeclercSchroeer10b} \cite{plamondon2013generic}
\cite{qin2019bases} \cite{musiker2013bases};
\item the common triangular bases in the sense of \cite{qin2017triangular},
see also \cite{BerensteinZelevinsky2012} \cite{qin2020dual,qin2023analogs};
\item the theta basis in the sense of \cite{gross2018canonical}, see also
\cite{davison2019strong} \cite{fock2006moduli} \cite{thurston2014positive}
\cite{MandelQin2021}.
\end{itemize}
The basis elements we are interested in are often $g$-pointed elements
for $g\in\Z^{I}$, i.e., they take the form $s_{g}=x^{g}\cdot F_{s_{g}}|_{y_{k}\mapsto x^{p^{*}n}}$,
where $F_{s_{g}}$ is a polynomial in the variables $y_{k}$, $k\in I_{\ufv}$,
$F_{s_{g}}(0)=1$. Then the freezing operator act on these $g$-pointed
elements by sending $y_{j}$ to $0$ for $j\in F$, see Section \ref{sec:Basics-of-the-freezing}. 

\begin{Thm}[{Theorem \ref{thm:freezing_basis}}]\label{thm:intro_freeze_basis}

For many $\upClAlg(t_{0})$, its basis of the form $\cS=\{s_{g}|g\in\Z^{I}\}$,
where $s_{g}$ are $g$-pointed, is sent to a basis of $\upClAlg(\frz_{F}t_{0})$
by the freezing operator.

\end{Thm}

Assume that we are in the situation of Theorem \ref{thm:intro_freeze_basis}.
We can also construct bases for $\upClAlg(\frz_{F}t_{0})$ by a different
method using localization, see Section \ref{subsec:Construction-by-localization}.
In brief, we consider the subset $\cS'\subset\cS$, which consists
of $s_{g}$ whose $F$-polynomial contains no variables $y_{j}$,
$j\in F$. Then we construct the localized function $x^{d}\cdot s_{g}$
for $d\in\Z^{F}$ and $s_{g}\in\cS'$. 

\begin{Thm}[{Theorem \ref{thm:induced_theta_basis}}]\label{thm:intro_induced_basis}

In the situation of Theorem \ref{thm:intro_freeze_basis}, under mild
conditions, the set of localized functions coincides with the basis
$\{\frz_{F}s_{g}|g\in\Z^{I}\}$ constructed via freezing.

\end{Thm}

The theta bases are sent to theta bases by the freezing operator,
see Theorem \ref{thm:freeze_theta_basis}. In Section \ref{sec:well-known-bases},
we will show that Theorem \ref{thm:intro_induced_basis} applies to
the theta bases. 

\begin{Rem}

One can also show that the common triangular bases are sent to the
common triangular bases, and the generic bases are sent to the generic
bases. Moreover, Theorems \ref{thm:intro_induced_basis} apply to
these bases as well. The claims for the common triangular bases are
verified in \cite{qin2023analogs}. And the claims for the generic
bases easily follow from \cite[Lemma 3.17]{plamondon2013generic}.

\end{Rem}

\begin{Rem}

The procedure that we take the subset $\cS'$ from $\cS$ has an analog
at the categorical level. 

In cluster category, its counterpart is taking a subcategory $\cC'$
from $\cC$ such that its object are extension orthogonal to the indecomposable
rigid objects $T_{j}$ corresponding to $x_{j}$, $j\in F$, see the
Calabi-Yau reduction \cite{IyamaYoshino08}. When $\cS$ is the generic
basis, $\cS'$ corresponds to the generic basis constructed from $\cC'$. 

In monoidal categorification, its counterpart is taking a monoidal
subcategory. More details are available in \cite{qin2023analogs}.

\end{Rem}

\subsection*{Convention}

Unless otherwise specified, we choose $\kk=\Z$ and $v=1$ for the
classical case, and $\kk=\Z[v^{\pm}]$ where $v$ is an indeterminate
for the quantum case.\footnote{One can study positive characteristic by considering the coefficients
modulo a prime $p$ (see \cite[Section 4.3]{mandel2021scattering}
for an example).} In literature, $v^{2}$ is often denoted by $q$. By nonnegative
elements in $\kk$, we mean those in $\Z_{\geq0}$ or $\Z_{\geq}[v^{\pm}]$
respectively.

For any index set $I'\subset I$, we often identify $\R^{I'}$ with
the subset $\{u=(u_{i})_{i\in I}\in\R^{I}|u_{j}=0\ \forall j\in I\backslash I'\}$
of $\R^{I}$. We also use $\pr_{I'}$ to denote the natural projection
from $\R^{I}$ to $\R^{I'}$.

\section*{Acknowledgments}

Some results have been announced in various occasions, for example in
\cite{Qin2017hereditary}, and the author would
like to thank the audience for the feedback.

\section{Preliminaries}

We recall some important constructions and theorems in cluster theory.

\subsection{Basics of cluster algebra\label{subsec:Basics-of-cluster}}

\subsubsection*{Seeds}

As in \cite{qin2019bases}, we shall define seeds of cluster algebras
following the convention of Fomin-Zelevinsky \cite{fomin2002cluster}
but include extra data that allows us to construct scattering diagrams
as in \cite{gross2018canonical}.

Let $I$ denote a finite set of vertices and we choose its partition
$I=I_{\ufv}\sqcup I_{\fv}$ into the frozen vertices and unfrozen
vertices. Let $d_{i}$, $i\in I$, denote a set of strictly positive
integers, and $d$ their least common multiplier. Denote the Langlands
dual $d_{i}^{\vee}=\frac{d}{d_{i}}$. Apparently, the least common
multiplier of $d_{i}^{\vee}$ is also $d$, and we have $(d_{i}^{\vee})^{\vee}=d_{i}$.

Let $(b_{ij})_{i,j\in I}$ denote a $\Q$-valued matrix such that
$d_{i}^{\vee}b_{ij}=-d_{j}^{\vee}b_{ji}$, and we say that it is skew-symmetrizable.
Define $\tB=(b_{ij})_{i\in I,j\in I_{\ufv}}$ and $B=(b_{ij})_{i,j\in I_{\ufv}}$.
The matrix $\tB$ is called the $B$-matrix and $B$ its principal
part. We further assume that $b_{ij}\in\Z$ if $i\in I_{\ufv}$ or
$j\in I_{\ufv}$. Following \cite{gross2013birational}, we also define
the skew-symmetric matrix $(\omega_{ij})_{i,j\in I}$ such that $\omega_{ij}d_{j}=b_{ji}$.

Let us define the linear map $p^{*}:\Z^{I_{\ufv}}\rightarrow\Z^{I}$
such that $p^{*}e_{k}=\sum_{i\in I}b_{ik}f_{i}$ where $e_{k}$ (resp
$f_{i}$) is the $k$-th unit vector (resp. $i$-th unit vector).
In terms of the column vectors of coordinates, we can also write $p^{*}e_{k}=\tB\cdot e_{k}$.
We will denote $\cone=\oplus_{i\in I}\Z f_{i}$ and $N_{\ufv}=\oplus_{k\in I_{\ufv}}\Z e_{k}$.

Let $(\ )^{T}$ denote the matrix transpose. Any $I\times I$ matrix
$\Lambda$ gives a bilinear form $\lambda$ on $\Z^{I}$ such that
$\lambda(g,h):=g^{T}\cdot\Lambda\cdot h$ for any $g,h$, 

Throughout the paper, we assume that $\tB$ is of full rank unless
otherwise specified. Then $p^{*}$ is injective. Then, by \cite{gekhtman2003cluster,GekhtmanShapiroVainshtein05},
we can always choose a (not necessarily unique) $\Z$-valued skew-symmetric
form $\lambda$ on $\Z^{I}$, called a compatible Poisson structure,
such that 
\begin{align}
\lambda(f_{i},p^{*}e_{k}) & =\delta_{ik}\diag_{k}\label{eq:compatible_pair}
\end{align}
 for strictly positive integers $\diag_{k}$. We denote $\Lambda=(\Lambda_{ij})_{i,j\in I}:=(\lambda(f_{i},f_{j}))_{i,j\in I}$,
called the quantization matrix. The pair $(\tB,\Lambda)$ is called
compatible by \cite{BerensteinZelevinsky05}.

\begin{Def}

A seed $t$ is a collection\footnote{As shown in \cite[Lemma 2.1.2]{qin2019bases}, we can always extend
$\tB$ to a matrix $(b_{ij})_{i,j\in I}$ with integer entries.} $(\tB,(x_{i})_{i\in I},I,I_{\ufv},(d_{i})_{i\in I})$, where $x_{i}$
are indeterminates called the $x$-variables. It is called a quantum
seed if we endow it with a compatible Poisson structure $\lambda$.

\end{Def}

We shall often omit the symbols $I,I_{\ufv},(d_{i})_{i\in I}$ when
the context is clear.

Let $\kk$ denote the base ring as before. We define the Laurent polynomial
ring $\LP=\kk[x_{i}^{\pm}]_{i\in I}$ using the usual commutative
product $\cdot$. Define $\cone:=\oplus\Z f_{i}$ and its group ring
$\kk[\cone]=\oplus_{m\in\cone}\kk x^{m}$. Then we can identify $\LP=\kk[x_{i}^{\pm}]$
with $\kk[\cone]$ by identifying $x_{i}$ with $x^{f_{i}}$. We call
$x_{i}$ the $x$-variables or cluster variables of $t$, the monomials
$x^{m}$, where $m\in\N^{I}$, the cluster monomial of $t$, and $x^{m}$,
$m\in\N^{I_{\ufv}}\oplus\Z^{I_{\fv}}$, the localized cluster monomials
of $t$. The $x$-variables $x_{j}$, $j\in I_{\fv}$, are called
the frozen variables. We also introduce $M^{\oplus}=p^{*}\N^{I_{\ufv}}$,
and let $\kk[M^{\oplus}]$ denote the subset of $\kk[\cone(t)]$.

We further endow $\LP$ with the following $v$-twisted product $*$

\begin{align*}
x^{g}*x^{h} & =v^{\lambda(g,h)}x^{g+h},\ \forall g,h\in\Z^{I}.
\end{align*}
Unless otherwise specified, by the multiplication in $\LP$, we mean
the $v$-twisted product. When we work at the classical level, any
$v$-factor is understood as $1$.

We will call $\LP$ the quantum torus algebra for the quantum case.
We define the bar involution on $\LP$ to be the anti-automorphism
such that $\overline{q^{\alpha}x^{g}}=q^{-\alpha}x^{g}$. Let $\cF$
denote the skew-field of fractions of $\LP$. The (partially) compactified
quantized Laurent polynomial ring is defined as the subalgebra $\bLP=\kk[x_{k}^{\pm}]_{k\in I_{\ufv}}[x_{j}]_{j\in I_{\fv}}$.

\subsubsection*{Mutations}

Let $[\ ]_{+}$ denote $\max(\ ,0)$. Given a seed $t=((b_{ij})_{i,j\in I},(x_{i})_{i\in I},I,I_{\ufv},(d_{i})_{i\in I})$.
Choose any vertex $k\in I$ and any sign $\varepsilon\in\{+,-\}$.
Following \cite{BerensteinFominZelevinsky05}, we define matrices
$\tE$ and $\tF$ as the following

\[
(\tE_{\varepsilon})_{ij}=\begin{cases}
\delta_{ij} & k\notin\{i,j\}\\
-1 & i=j=k\\{}
[-\varepsilon b_{ik}]_{+} & j=k,i\neq k
\end{cases}
\]

\[
(\tF_{\varepsilon})_{ij}=\begin{cases}
\delta_{ij} & k\notin\{i,j\}\\
-1 & i=j=k\\{}
[\varepsilon b_{kj}]_{+} & i=k,j\neq k
\end{cases}
\]
We have $\tE_{\varepsilon}^{2}=\tF_{\varepsilon}^{2}=\Id_{I}$. Their
principal parts ($I_{\ufv}\times I_{\ufv}$-submatrices) are denoted
by $E_{\varepsilon}$ and $F_{\varepsilon}$ respectively.

We have an operation $\mu_{k}$ called mutation, which gives us a
new seed $t'=((b_{ij}')_{i,j\in I},(x_{i}')_{i\in I},I,I_{\ufv},(d_{i})_{i\in I})$
as follows

\begin{align}
(b_{ij}')_{i,j\in I} & =\tE_{\varepsilon}\cdot(b_{ij})_{i,j}\cdot\tF_{\varepsilon}\nonumber \\
x_{i}' & =\begin{cases}
x_{k}^{-1}\cdot(\prod_{i\in I}x_{i}^{[b_{ik}]_{+}}+\prod_{j\in I}x_{j}^{[-b_{jk}]_{+}}) & i=k\\
x_{i} & i\neq k
\end{cases}\label{eq:exchange_relation}
\end{align}

If we denote $z_{i}=x^{p^{*}e_{i}}$, then we have

\begin{align}
z_{i}' & =\begin{cases}
z_{k}^{-1} & i=k\\
z_{i}z_{k}^{[b_{ki}]_{+}}\cdot(1+z_{k})^{-b_{ki}} & i\neq k
\end{cases}.\label{eq:mutation-y}
\end{align}

If $t$ is a quantum seed, the mutation $\mu_{k}$ also gives us a
new quantum seed $t'$ whose quantization matrix $\Lambda'$ is given
by $\Lambda'=\tE_{\varepsilon}^{T}\cdot\Lambda\cdot\tE_{\varepsilon}$.
See \cite[Section 2.4]{kimura2022twist} for the quantum analogs of
(\ref{eq:exchange_relation}) (\ref{eq:mutation-y}).

Direct computation shows that that 
\begin{align*}
b_{ij}' & =\begin{cases}
-b_{ij} & k\in\{i,j\}\\
b_{ij}+[b_{ik}]_{+}[b_{kj}]_{+}-[-b_{ik}]_{+}[-b_{kj}]_{+} & i,j\neq k
\end{cases}.
\end{align*}
Moreover, $\Lambda'$ not depend on $\varepsilon$ by \cite[Proposition 3.4]{BerensteinZelevinsky05}.
Therefore, the new (quantum) seed $t'$ is independent of the choice
of the sign $\varepsilon$.

Note that $\mu_{k}$ is an involution: $\mu_{k}(\mu_{k}t)=t$. Let
$\mu_{k}^{*}$ denote the $\kk$-algebra homomorphism from the (skew-)field
of fractions $\cF(t')$ to $\cF(t)$ such that the image $\mu_{k}^{*}(x_{i}')$
is given by (\ref{eq:exchange_relation}). Then $\mu_{k}^{*}$ is
an isomorphism. We often identify $\cF(t')$ with $\cF(t)$ via $\mu_{k}^{*}$
and omit the symbol $\mu_{k}^{*}$ for simplicity.

For any sequence $\uk=(k_{1},\ldots,k_{r})$ on $I_{\ufv}$, we define
the mutation sequence $\seq_{\uk}=\mu_{k_{r}}\cdots\mu_{k_{2}}\mu_{k_{1}}$,
which is often denoted by $\seq$.

Let $\sT_{v_{0}}$ denote an $I_{\ufv}$-regular tree with root $v_{0}$,
namely, a regular tree such that for each vertex the adjacent edges
are labeled by $I_{\ufv}$. Given an initial (quantum) seed $t_{0}$,
we recursively attach a (quantum) seed $t_{v}$ to each vertex $v\in\sT_{v_{0}}$
as follows:

\begin{enumerate}

\item The (quantum) seed $t_{v_{0}}$ attach to $t_{0}$ is the initial
(quantum) seed $t_{0}$.

\item If $v$ and $v'$ are connected by an edge labeled $k\in I_{\ufv}$,
then the attached seeds satisfy $t_{v'}=\mu_{k}t_{v}$.

\end{enumerate}

Let $\Delta^{+}=\Delta_{t_{0}}^{+}$ denote the set of (quantum) seeds
$t$ that we obtain. The $x$-variables $x_{i}(t)$ associated to
$t\in\Delta^{+}$ are called the (quantum) cluster variables in the
seed $t$. Note that the $x$-variables $x_{j}(t)$, $j\in I_{\fv}$,
are the same for all $t\in\Delta^{+}$, and we call them the frozen
variables. 

\begin{Def}

Let there be given an initial seed $t_{0}$. The (partially compactified)
cluster algebra $\bClAlg(t_{0})$ is defined to be the $\kk$-algebra
generated by the cluster variables $x_{i}(t)$. The (localized) cluster
algebra is defined to be the localization $\clAlg(t_{0})=\bClAlg(t_{0})[x_{j}^{-1}]_{j\in I_{\fv}}$.
The (localized) upper cluster algebra $\upClAlg(t_{0})$ is defined
to be the intersection $\cap_{t\in\Delta^{+}}\LP(t)$.

\end{Def}

Throughout the paper, we shall focus our attention to (localized)
upper cluster algebras $\upClAlg$.

\begin{Thm}[Starfish theorem {\cite{BerensteinFominZelevinsky05}}]\label{thm:starfish_thm}

When $\tB(t)$ is of full rank, we have $\upClAlg(t)=\LP(t)\bigcap(\cap_{k\in I_{\ufv}}\LP(\mu_{k}t)).$

\end{Thm}

\subsection{Cluster expansions}

\begin{Thm}[{Laurent phenomenon \cite{fomin2002cluster}\cite{BerensteinZelevinsky05}}]

The (quantum) cluster algebra $\clAlg(t_{0})$ is contained in the
(quantum) upper cluster algebra $\upClAlg(t_{0})$.

\end{Thm}

In particular, any (quantum) cluster variable $x_{i}(t)$ is contained
in $\LP(t_{0})$. We can describe its Laurent expansion as follows.

\begin{Thm}\cite{FominZelevinsky07}\cite{DerksenWeymanZelevinsky09}\cite{Tran09}\cite{gross2018canonical}\label{thm:cluster_expansion}

Let there be any initial seed $t_{0}$. For any $i\in I$, $v\in\sT_{I_{\ufv}}$,
there exist a vector $g_{i}(t_{v})\in\Z^{n}$ and coefficients $b_{n}\in\kk$,
such that $b_{0}=1$, all but finitely many $b_{n}=0$, and 

\begin{align}
x_{i}(t_{v}) & =x^{g_{i}(t_{v})}\cdot\sum_{n\in\N^{I_{\ufv}}}b_{n}x^{p^{*}n}.\label{eq:cluster_expansion}
\end{align}

\end{Thm}

The vector $g_{i}(t_{v})$ in (\ref{eq:cluster_expansion}) is called
the $i$-th (extended) \emph{$g$-vector} associated to $t_{v}$.
Note that if the two initial seeds in Theorem \ref{thm:cluster_expansion}
share the same $B$-matrix $\tB=(b_{ij})_{i\in I,j\in I_{\ufv}}$,
then we have the same vectors $g_{i}(t_{v})$ in (\ref{eq:cluster_expansion}).
Moreover, by evaluating $v=1$, quantum cluster variables become the
corresponding classical cluster variables. See \cite{Tran09}\cite{geiss2020quantum}
for more details about the relation between quantum and commutative
cluster algebras.

\begin{Thm}[{\cite{gross2018canonical}\cite{davison2016positivity}}]\label{thm:positivity_cluster}

(1) When $\kk=\Z$, the coefficients $b_{n}$ in \ref{eq:cluster_expansion}
belong to $\N$.

(2) When $\kk=\Z[q^{\pm\Hf}]$, if $B$ is skew-symmetric, then the
coefficients $b_{n}$ in \ref{eq:cluster_expansion} belong to $\N[q^{\pm\Hf}]$.

\end{Thm}

\begin{Def}[{\cite{qin2017triangular}\cite{qin2019bases}}]

We call $t_{0}$ injective-reachable if there exists a seed $t_{0}[1]\in\Delta^{+}$
such that, for some permutation $\sigma$ of $I_{\ufv}$, we have
$g_{\sigma k}(t_{0}[1])\in-f_{k}+\Z^{I_{\fv}}$, for any $k\in I_{\ufv}$.

\end{Def}

If $t_{0}$ is injective-reachable, then so are all seeds $t\in\Delta^{+}$,
see \cite{muller2015existence}\cite[Proposition 5.1.4]{qin2017triangular}.

\subsection{Dominance order, degrees, and support}

Let there be given a seed $t$. Omit the symbol $t$ for simplicity
when the context is clear. We have lattices $\cone(t)\simeq\Z^{I}$
with the natural basis $\{f_{i}|i\in I\}$ and $\yCone(t)\simeq\Z^{I_{\ufv}}$
with the natural basis $\{e_{i}|i\in I_{\ufv}\}$. Denote $\yCone^{\geq0}(t)=\oplus_{i\in I_{\ufv}}\N e_{i}$
and $\yCone^{>0}(t)=\yCone^{\geq0}(t)\backslash\{0\}$.

\begin{Def}[{Dominance order \cite{qin2017triangular}}]

Given any $g,g'\in\cone(t)$. We say $g$ dominates $g'$, denoted
by $g'\preceq_{t}g$, if and only if $g'=g+p^{*}n$ for some $n\in\yCone^{\geq0}(t)$.
We write $g'\prec_{t}g$ if further $g'\neq g$.

\end{Def}

\begin{Lem}[{\cite[Lemma 3.1.2]{qin2017triangular}}]

For any $\eta,g\in\cone(t)$, the set $\{g'\in\cone(t)|\eta\preceq_{t}g'\preceq_{t}g\}$
is finite.

\end{Lem}

Recall that we have a subalgebra $\kk[M^{\oplus}(t)]\subset\LP(t)$.
It has a maximal ideal generated by $x^{p^{*}e_{k}}$, $k\in I_{\ufv}$.
Denote the corresponding completion by $\widehat{\kk[M^{\oplus}(t)]}$.
We define the set of formal Laurent series to be

\begin{align*}
\hLP(t) & =\LP(t)\otimes_{\kk[M^{\oplus}(t)]}\widehat{\kk[M^{\oplus}(t)]}.
\end{align*}
Then any formal Laurent series is a finite sum of the type $a\cdot x^{g}\cdot\sum_{n\in\yCone^{\geq0}(t)}b_{n}x^{p^{*}n}$,
for $a,b_{n}\in\kk$.

\begin{Def}[Degree]

Let there be given any formal sum $z=\sum c_{m}x^{m}$, $c_{m}\in\kk$.
If the set $\{m|c_{m}\neq0\}$ has a unique $\prec_{t}$-maximal element
$g$, we say $z$ has degree $g$ and denote $\deg^{t}z=g$. We further
say that $z$ is $g$-pointed (or pointed at $g$) if $c_{g}=1$.

\end{Def}

We define $\hLP_{\preceq g}(t)$ to be the subset $x^{g}*\widehat{\kk[M^{\oplus}(t)]}$
of $\hLP(t)$. Then we have $\hLP_{\preceq g}(t)=\sum_{g'\preceq_{t}g}\ptSet^{t}(g')$.

The following definition slightly generalizes that of \cite[Definition 3.4.1]{qin2019bases}.

\begin{Def}[Support]

For any $n=(n_{k})_{k\in I_{\ufv}}\in\yCone(t)\simeq\Z^{I_{\ufv}}$,
we define its support to be $\supp n=\{k\in I_{\ufv}|n_{k}\neq0\}$.

For any $z=x^{g}\cdot\sum_{n\in\yCone^{\geq0}}b_{n}x^{p^{*}n}$ in
$\hLP(t)$, where $b_{0}\neq0$, $b_{n}\in\kk$. We define its support
to be the union of $\supp n$ for $b_{n}\neq0$. If further $z\in\LP(t)$,
we define its support dimension to be $\suppDim z=\sum_{k\in I_{\ufv}}\max_{b_{n}\neq0}(n_{k})\cdot e_{k}$.

\end{Def}

Finally, we give the following definition to facilitate later discussion.

\begin{Def}[Pointed subset]

Let $\Theta$ denote a subset of $\Z^{I}$ and $\cS=\{s_{g}|g\in\Theta\}$
a subset of $\hLP(t_{0})$. If $s_{g}$ are $g$-pointed, $\cS$ is
said to be $\Theta$-pointed.

\end{Def}

\subsection{Tropical transformation and linear transformation}

Recall that, to each seed $t$, we associate a lattice $\cone(t)\simeq\Z^{I}$
with the natural basis $f_{i}(t)$. For any $k\in I_{\ufv}$, denote
$t'=\mu_{k}t$ . The tropical transformation $\phi_{t',t}:\cone(t)\ra\cone(t')$
is the piecewise linear map such that, for any $m=\sum m_{i}f_{i}(t)\in\cone(t)$,
its image $m'=\sum m_{i}'f_{i}(t')\in\cone(t')$ is given by:

\begin{align}
m_{i}'= & \begin{cases}
-m_{k} & i=k\\
m_{i}+b_{ik}(t)[m_{k}]_{+} & b_{ik}(t)\geq0\\
m_{i}+b_{ik}(t)[-m_{k}]_{+} & b_{ik}(t)\leq0
\end{cases}\label{eq:tropical_transformation}
\end{align}

For any mutation sequence $\seq$ and the seed $t'=\seq t$ in $\Delta^{+}$,
we define $\phi_{t',t}$ to be the composition of the tropical transformations
along $\seq$. It is independent of the choice of $\seq$ by the following
result.

\begin{Thm}[{\cite{gross2013birational}\cite{gross2018canonical}}]\label{thm:compatibly_pointed_cluster_monom}

For any $i\in I$ and seeds $t,t',t''\in\Delta^{+}$, we have $\deg^{t''}x_{i}(t)=\phi_{t'',t'}\deg^{t'}x_{i}(t)$.

\end{Thm}

Given two seeds $t,t'\in\Delta^{+}$ related by a mutation sequence
$t=\seq t'$. Following \cite{qin2019bases}, we also define a $\Z$-linear
transformation $\psi_{t',t}$ from $\cone(t)$ to $\cone(t')$ such
that $\psi_{t',t}(f_{i}(t))=\deg^{t'}x_{i}(t)$. Then it is a a bijection.

\begin{Lem}[\cite{qin2019bases}]

Given any $z=\sum c_{m}x^{m}(t)\in\LP(t)$. Then the formal Laurent
series $\seq^{*}z$ has degree $g\in\cone(t')$ if and only if $g$
is the unique $\prec_{t'}$-maximal element in $\{\psi_{t',t}m|c_{m}\neq0\}$.
$\seq^{*}z$ is further pointed at degree $g$ if and only if $c_{m}=1$
where $m=\psi_{t',t}^{-1}g$.

\end{Lem}

\begin{Def}[Compatibly pointed \cite{qin2019bases}]

Given an $m$-pointed element $z\in\LP(t)$, it is said to be compatibly
pointed at $t,t'$ if $z$ is $\phi_{t',t}m$-pointed in $\LP(t')$.
A subset $Z\subset\LP(t)$ is said to be compatibly pointed at $t,t'$
if its elements are.

\end{Def}

In particular, Theorem \ref{thm:compatibly_pointed_cluster_monom}
implies that the localized cluster monomials are compatibly pointed
at all seeds in $\Delta^{+}$.

\section{A review of classical scattering diagrams \label{sec:Scattering-diagram}}

For classical cases, we briefly review the scattering diagrams and
some of their properties following \cite{gross2018canonical}. More
details could be found in \cite{bridgeland2017scattering}\cite{gross2018canonical}.
When $B(t_{0})$ is skew-symmetric, quantum cluster scattering diagrams
have similar but more technical descriptions, see \cite{davison2019strong}
\cite[Section 5]{MandelQin2021}.

We choose and fix an initial seed $t_{0}=((b_{ij})_{i,j\in I},(x_{i})_{i\in I},I,I_{\ufv},(d_{i})_{i\in I})$
and omit the symbol $t_{0}$ for simplicity when the context is clear.
Define a lattice $N\simeq\Z^{I}$ with the natural basis $\{e_{i}|i\in I\}$.
Its dual lattice $M=\Hom_{\Z}(N,\Z)$ possesses the dual basis $\{e_{i}^{*}|i\in I\}$.
The natural pairing between elements $n\in N$ and $m\in M$ are denoted
by $\langle m,n\rangle=m(n)=n(m)$. Define $f_{i}=\frac{1}{d_{i}}e_{i}^{*}$,
$\cone=\oplus_{i\in I}\Z f_{i}$, $N^{\geq0}=\oplus_{i\in I}\N e_{i}$,
$N^{\circ}=\oplus_{i\in I}\Z d_{i}e_{i}$, $\yCone=\oplus_{k\in I_{\ufv}}\Z e_{k}$.
We naturally extend $\cone$ to the affine space $\cone_{\R}=\cone\otimes_{\Z}\R=M\otimes_{\Z}\R=M_{\R}$.

Recall that $(\omega_{ij})=(b_{ji}d_{j}^{-1})$ is skew-symmetric.
Endow $N$ with the skew-symmetric $\Q$-valued bilinear form $\{\ ,\ \}$
such that $\{e_{i},e_{j}\}=\omega_{ij}$.

Choose the base ring $\kk=\Q$. Let $y_{i}$ be indeterminate, $i\in I$.
We have two monoid algebras $\kk[\yCone^{\geq0}]:=\oplus_{n\in\yCone^{\geq0}}\kk y^{n}$
and $\kk[\cone]:=\oplus_{m\in M}\kk x^{m}$, where $y_{i}=y^{e_{i}}$
and $x_{i}=x^{f_{i}}$. Let $\widehat{\kk[\yCone^{\geq0}]}$ denote
the completion of $\kk[\yCone^{\geq0}]$ with respect to its maximal
ideal $\kk[\yCone^{>0}]$ as before.

Recall that we have a map $p^{*}:\yCone\rightarrow\cone$ such that
$p^{*}e_{k}=\sum b_{ik}f_{i}$. Note that we have $p^{*}(n)=\{n,\ \}$.
We can send $\kk[\yCone^{\geq0}]$ into $\kk[\cone]$ via the associated
$\kk$-linear monomial map $p^{*}$ such that $p^{*}y^{n}:=x^{p^{*}n}$.

\subsection{Scattering diagrams associated to nilpotent graded Lie algebras}

Let $\frg$ be any given $\yCone^{>0}$-graded nilpotent Lie algebra
whose Lie bracket is denoted by $[\ ,\ ]$:
\begin{align*}
\frg & =\oplus_{n\in\yCone^{>0}}\frg_{n}\\{}
[\frg_{n_{1}},\frg_{n_{2}}] & \subset\frg_{n_{1}+n_{2}},\quad\forall n_{1},n_{2}.
\end{align*}
Then we construct the corresponding pro-nilpotent group $G$ as the
following: it is in bijection with $\frg$ as a set
\begin{align*}
\exp:\frg & \simeq G,
\end{align*}
and its multiplication is given by the Baker-Campbell-Hausdorff formula,
namely, $\forall\exp(X),\exp(Y)\in G$, we have
\begin{align*}
 & \exp(X)\cdot\exp(Y)\\
 & =\exp(X+Y+\frac{1}{2}[X,Y]+\frac{1}{12}[X,[X,Y]]+\frac{1}{12}[Y,[X,Y]]+\cdots)
\end{align*}

\subsubsection*{Cones}

By a \emph{cone} $\sigma$ in $\cone_{\R}=M_{\R}$, we always mean
a convex rational polyhedral cone, namely, a conic combination of
finitely many generators $\{m_{i}\}\subset\cone$: $\sigma=\sum_{i}\R_{\geq0}m_{i}$.
Note that it can also be described as $\sigma=\cap_{j}H_{n_{j}}^{\geq0}$
for finitely many $n_{j}\in N$, where the subset $H_{n_{j}}^{\geq0}:=\{m|m(n_{j})\geq0\}$.
Then it is clear that the intersection of two cones remains a cone.
A cone $\sigma$ will be called \emph{strictly convex} if $x=0$ whenever
$x,-x\in\sigma$.

The codimension of $\sigma$ is defined to be that of $\sigma\otimes\R=\sum\R m_{i}$
in $\cone_{\R}$. Given a codimension $1$ cone $\frd$. Then $\{n\in N_{\R}|n(\frd)=0\}$
is one dimensional, whose non-zero elements are called the normal
vectors of $\frd$. There exists a unique primitive normal vector
$n_{0}$ up to a sign. A point $x$ in $M_{\R}$ is said to be \emph{general},
if it is contained in at most one hyperplane $n^{\bot}$ for some
$n\in N$.

A face of a cone $\sigma$ is a conic combination of some generators.
It is again a cone. We use $\Int(\sigma)$ to denote the relative
interior of $\sigma$, namely, those points not belonging to any face
of $\sigma$. Let $\partial\frd$ denote the boundary of $\frd$.

\subsubsection*{Scattering diagrams}

For any $n\in\yCone^{>0}$, we have the following sub-Lie algebra
of $\frg$ and the corresponding subgroup of $G$:

\begin{align*}
\frg_{n}^{\|} & :=\oplus_{k>0}\frg_{k\cdot n}\\
G_{n}^{\|} & :=\exp\frg_{n}^{\|}
\end{align*}

\begin{Def}

A wall $(\frd,\frp_{\frd})$ consists of a codimension $1$ cone $\frd\subset n_{\frd}^{\perp}$
for some primitive normal vector $n_{\frd}\in N_{\ufv}^{>0}$ and
a group element $g\in G_{n_{\frd}}^{\bigparallel}$. We call $\frd$
the support of the wall and $\frp_{\frd}$ the wall crossing operator.

\end{Def}

We often call the support $\frd$ a wall for simplicity. We make the
following assumption from now on.

\begin{Assumption}

We always assume that $\frg_{n}^{\|}$ is commutative for any $n\in\yCone^{>0}$.

\end{Assumption}

It follows that the subgroup $G_{n}^{\|}$ is abelian.

\begin{Def}[{Finite scattering diagram \cite[Definition 1.6]{gross2018canonical}\cite[Definition 2.1]{bridgeland2017scattering}}]

Let $\frg$ be an $\yCone^{>0}$-graded nilpotent Lie algebra. A $\frg$-scattering
diagram $\frD$ is a finite collection of walls $(\frd,\frp_{\frd})$,
such that $\frd\subset n_{\frd}^{\bot}$ for some primitive $n_{\frd}\in\yCone^{>0}$.

\end{Def}

If the context is clear, we will not mention the Lie algebra $\frg$
when referring to the scattering diagram $\frD$. 

\begin{Def}

The support of $\frD$ is the union $\supp\frD=\cup_{\frd\in\frD}\frd$.
The singular locus of $\frD$ is defined to be

\begin{align*}
\sing\frD=(\cup_{\frd\in\frD}\partial\frd)\cup(\cup_{\frd_{1},\frd_{2}\in\frD,\codim(\frd_{1}\cap\frd_{2})=2}\frd_{1}\cap\frd_{2}).
\end{align*}

\end{Def}

A smooth path $\gamma:[0,1]\ra M_{\R}$ is said to be $\frD$-generic,
if

\begin{itemize}

\item then end points $\gamma(0)$, $\gamma(1)$ are not in $\supp\frD$,

\item $\gamma$ does not intersect with the singular locus $\sing\frD$,
and

\item the intersections between $\gamma$ and the walls $\frd\in\frD$
are transversal.

\end{itemize}

It is said to be closed if $\gamma(0)=\gamma(1)$.

Let there be given a $\frD$-generic smooth path $\gamma(t)$, such
that it intersects the walls $\frd_{1},\ldots\frd_{l}$ at time $0<t_{1}\le\cdots\leq t_{l}<1$.
Note that when $t_{i}=t_{i+1}$, the wall $\frd_{i}$ and $\frd_{i+1}$
must overlap at the transversal intersection point $\gamma(t_{i}$),
indicating that they share a primitive normal vector $n_{0}$. In
this case, the product $\frp_{\frd_{i+1}}\cdot\frp_{\frd_{i}}$ equals
$\frp_{\frd_{i}}\cdot\frp_{\frd_{i+1}}$ by the commutativity assumption. 

\begin{Def}

The wall crossing operator along the path $\gamma$ is defined to
be

\begin{align*}
\frp(\gamma) & =\frp(\frd_{l})^{\epsilon_{l}}\cdot\cdots\cdot\frp(\frd_{1})^{\epsilon_{1}}\in G,
\end{align*}
where $\epsilon_{i}=\sign\langle-\gamma'(t_{i}),n_{i}\rangle$, $n_{i}\in\yCone^{>0}$
are the primitive normal vectors of $\frd_{i}$.

\end{Def}

\begin{Def}[{Equivalent scattering diagram \cite[Definition 2.4]{bridgeland2017scattering}\cite[Definition 1.8]{gross2018canonical}}]

Two $\frg$-scattering diagrams $\frD_{1}$, $\frD_{2}$ are said
to be \emph{equivalent} if they give the same wall crossing operator
$\frp(\gamma)$ for any $\frD_{1}$-generic and $\frD_{2}$-generic
smooth path $\gamma$.

A scattering diagram $\frD$ is said to be consistent if $\frp(\gamma)$
is trivial for any closed $\frD$-generic smooth path $\gamma$.

\end{Def}

Following \cite{gross2018canonical}, we define the essential support
of $\frD$ to be the union $\supp{}_{\mathrm{{ess}}}\frD=\cup_{\frd\in\frD,\frp_{\frd}\neq1}\frd$.
Replacing $\frD$ by an equivalent scattering diagram if necessary,
we can assume the support to be minimal.

\subsubsection*{Scattering diagrams associated to graded Lie algebras}

Given an $\yCone^{>0}$-graded Lie algebra $\frg=\oplus_{n\in\yCone^{>0}}\frg_{n}$,
not necessarily nilpotent. We generalize the previous construction
of scattering diagrams for $\frg$.

Choose a linear function $d$ on $\yCone$ such that $d(e_{k})\in\Z_{>0}$
for $k\in I_{\ufv}$. Then, for any chosen order $k\in\N$, $\{n\in\yCone^{>0}|d(n)\leq k\}$
is a finite set. We have the subalgebra $\frg^{>k}$ and a nilpotent
Lie algebra $\frg^{\leq k}$:
\begin{align*}
\frg^{>k} & =\oplus_{d(n)>k}\frg_{n}\\
\frg_{}^{\leq k} & =\frg/\frg^{>k}
\end{align*}

Correspondingly, we have well defined pro-nilpotent groups at all
orders $k$ as before
\begin{align*}
G^{\leq k} & =\exp\frg_{}^{\leq k}
\end{align*}

By using the canonical surjection $\pi^{ji}:\frg^{\leq j}\ra\frg^{\leq i}$,
$\pi^{ji}:G^{\leq j}\ra G^{\leq i}$, for any $i\leq j$, $n\in\yCone^{>0}$,
we obtain the inverse limits 
\begin{align*}
\hat{\frg} & =\lim_{\longleftarrow}\frg^{\leq k}\\
\hG & =\lim_{\longleftarrow}G^{\leq k}\\
\hG_{n}^{\bigparallel} & =\lim_{\leftarrow}(G^{\leq k})_{n}^{\bigparallel}
\end{align*}

The set theoretic bijection $\exp:\frg^{\leq k}\simeq G^{\leq k}$
induces a bijection $\exp:\hat{\frg}\simeq\hG$.

As before, a wall $(\frd,\frp_{\frd})$ consists of a codimension
$1$ cone $\frd$ with some primitive normal vector $n_{0}\in\yCone^{\geq0}$,
and a wall crossing operator $\frp_{\frd}\in\hG_{n_{0}}^{\bigparallel}$.

Let $\pi^{\leq k}$ denote the natural projection from $\hG$ to $\hG^{\leq k}$.

\begin{Def}[Scattering diagram]

Let $\frg$ be an $\yCone^{>0}$-graded Lie algebra. A $\frg$-scattering
diagram $\frD$ is a collection of walls $(\frd,\frp_{\frd})$ such
that for any $k\in\N$, there are only finitely many nontrivial walls
$(\frd,\pi^{\leq k}\frp_{\frd})$.

\end{Def}

Let $\frD^{\leq k}$ denote the collection of the nontrivial walls
$(\frd,\pi^{\leq k}\frp_{\frd})$ obtained from a scattering diagram
$\frD$. Then it is a $\frg^{\leq k}$-scattering diagram. A smooth
path $\gamma$ is said to be $\frD$-generic if it is generic in all
$\frD^{\leq k}$. The wall crossing operator $\frp_{\gamma}$ along
a $\frD$-generic smooth path $\gamma$ is defined to be the inverse
limit of the wall crossing operators $\frp_{\gamma}^{\leq k}$ defined
in $\frD^{\leq k}$. As before, the scattering diagram $\frD$ is
said to be consistent if $\frp_{\gamma}$ is trivial for all closed
paths.

The essential support of $\frD$ is defined to be $\supp_{\ess}\frD=\cup_{k\geq1}\supp_{\ess}(\frD^{\leq k})$.
We often write $\frD=\cup\frD^{\leq k}$.

\subsubsection*{Pushforward of scattering diagrams}

Given a homomorphism $\pi:\frg\ra\frh$ between two $\yCone^{\geq0}$-graded
Lie algebras. It induces Lie algebra homomorphisms $\pi:\frg^{\leq k}\rightarrow\frh^{\leq k}$
and $\pi:\hfrg\rightarrow\hfrh$.

\begin{Lem}\label{lem:exp_intertwine_morphism}

(1) For any $k\in\N$, the map $\pi:\exp\frg^{\leq k}\ra\exp\frh^{\leq k}$
sending $\exp g$ to $\exp\pi(g)$ is a group homomorphism.

(2) The map $\pi:\exp\hfrg\ra\exp\hfrh$ sending $\exp g$ to $\exp\pi(g)$
is a group homomorphism.

\end{Lem}

\begin{proof}

(1) The map preserves the multiplication because the Lie algebra homomorphism
$\pi$ commutes with the Baker--Campbell--Hausdorff formula for
nilpotent Lie algebras.

(2) The claim follows from claim (1).

\end{proof}

Following \cite[Sectioni 2.10]{bridgeland2017scattering}, for any
$\frg$-scattering diagram $\frD$, we define the $\frh$-scattering
diagram $\pi_{*}\frD$ as the collection of the pairs $(\frd,\pi\frp_{\frd})$
for any $(\frd,\frp_{\frd})\in\frD$.

\subsection{Determination of scattering diagrams }

Let there be given an $\yCone^{>0}$-graded Lie algebra $\frg$ and
a consistent $\frg$-scattering diagram $\frD$. Following \cite{bridgeland2017scattering},
we refer to the connected components of $M_{\R}\backslash\overline{\supp_{\ess}\frD}$
as chambers associated to $\frD$, where $\overline{\supp_{\ess}\frD}$
denotes the closure.

We define the positive and the negative chambers of $M_{\R}$ by:

\begin{align*}
\frC^{+} & =\{m\in M_{\R}|\langle m,e_{k}\rangle\geq0,k\in I_{\ufv}\}\\
\frC^{-} & =\{m\in M_{\R}|\langle m,e_{k}\rangle\leq0,k\in I_{\ufv}\}.
\end{align*}
Then, for any codimension $1$ cone $\frd\subset n^{\bot}$ with $n\in\yCone^{>0}$,
we have $\Int(\frC^{+})\cap\frd=\emptyset$, $\Int(\frC^{-})\cap\frd=\emptyset$.
Choosing any generic smooth path $\gamma$ from $\frC^{+}$to $\frC^{-}$,
we define $\frp_{\frD}=\frp_{\gamma}$.

\begin{Thm}\cite[Proposition 3.4]{bridgeland2017scattering}\cite[Theorem 1.17]{gross2018canonical}\cite[Theorem 2.1.6]{kontsevich2014wall}

We have a bijection between the equivalent classes of consistent $\frg$-scattering
diagrams and the elements of $\hat{G}$, sending a scattering diagram
$\frD$ to $\frp_{\frD}$.

\end{Thm}

\begin{Def}\cite[Definition 1.11]{gross2018canonical}

A wall $(\frd,\frp_{\frd})$ with the primitive normal vector $n_{0}\in\yCone^{>0}$
is said to be incoming, if it contains $p^{*}n_{0}$, and outgoing
otherwise. We call $-p^{*}n_{0}$ the direction of the wall.

\end{Def}

\begin{Thm}\cite[Theorem 1.21(4)]{gross2018canonical}

The equivalent class of a consistent scattering diagram $\frD$ is
determined by its set of incoming walls $\frD_{\mathrm{in}}$.

\end{Thm}

\begin{Thm}\cite[Theorem 1.21(3) Proposition 1.20]{gross2018canonical}

Let $\frD_{\mathrm{in}}$ denote a scattering diagrams consisting
of incoming walls of the form $(n_{0}^{\bot},\frp_{n_{0}^{\bot}})$
with primitive $n_{0}\in\yCone^{\geq0}$. Then there exists a consistent
scattering diagram $\frD$ containing $\frD_{\mathrm{in}}$ such that
$\frD\backslash\frD_{\mathrm{in}}$ consists of outgoing walls.

\end{Thm}

We say a wall $(\frd,\frp_{\frd})$ is translation invariant in the
frozen direction if $\oplus_{j\in I_{\fv}}\R f_{j}\subset\frd$. By
the inductive construction of scattering diagrams \cite[Appendiex C.1]{gross2018canonical},
if the incoming walls of a consistent scattering diagram $\frD$ are
translation invariant in the frozen direction, we can find a $\frD'$
equivalent to $\frD$ such that all of its walls have this property.

\subsection{Cluster scattering diagram}

We denote the $\yCone^{\geq0}$-graded monoid algebra $A=\kk[\yCone^{\geq0}]=\kk[y_{k}]_{k\in I_{\ufv}}$
and its completion $\hA=\widehat{\kk[\yCone^{\geq0}]}$. Endow $A$
with a Poisson structure

\begin{align*}
\{y^{n},y^{n'}\} & =-\{n,n'\}y^{n+n'}.
\end{align*}
Define the corresponding $\yCone^{>0}$-graded Lie algebra $\frg=A_{>0}=\oplus_{n\in\yCone^{>0}}\kk y^{n}$,
where we take Lie bracket to be the Poisson bracket. Let $\hfrg$
denote the completion of $\frg$ by taking the inverse limit as before.
Let $\exp$ denote the standard Taylor series expansion and define
the corresponding pro-nilpotent group $\hG=\exp\hfrg$. Note that
$\hG$ is identified with the set $1+\hfrg\subset\hA$.

We can write any $a\in\hA$ as a formal series $a=\sum_{n\in\yCone^{\geq0}}a_{n}$
with homogeneous component $a_{n}\in\kk y^{n}$. For any primitive
$n_{0}\in\yCone^{>0}$, any $g=\sum_{k\geq1}g_{kn_{0}}\in\frg_{n_{0}}^{\bigparallel}$,
and $a\in\hA$, the natural action of $g$ on $a$ is given by the
derivation $\{g,a\}=\sum_{k\geq1}\sum_{n}\{g_{kn_{0}},a_{n}\}$. The
natural action of $\exp g$ on $a$ is given by $(\exp g)(a)=\exp(\{g,\})(a)$,
where $\exp$ is the Taylor series expansion.

Let $z$ denote an indeterminate. Recall that the dilogarithm function
$\Li_{2}(z)$ of $z$ is the following formal series
\begin{align*}
\Li_{2}(z)=\sum_{k\geq1}\frac{z^{k}}{k^{2}}.
\end{align*}

\begin{Def}[Cluster scattering diagram]

Let $t_{0}$ denote a given initial seed. The\emph{ }cluster scattering
diagram $\frD(t_{0})$ is a consistent $\frg$-scattering diagram
whose set of incoming walls is $\{(e_{k}^{\bot},\exp(-d_{k}\Li_{2}(-y_{k}))|k\in I\}$,
such that it has minimal support.

\end{Def}

\begin{Def}[Reachable chamber]\label{def:reachable_chamber}

Let $\frC,\ \frC'$ be two chambers of a scattering diagram $\frD$.
$\frC$ is said to be reachable from $\frC'$ if there exists a generic
smooth path from $\Int(\frC')$ to $\Int(\frC)$ such that it intersects
finitely many walls.

When $\frC$ is reachable from the initial chamber $\frC^{+}$, it
is simply called (mutation) reachable.

\end{Def}

\begin{Thm}\cite[Section 1.3 Lemma 2.10]{gross2018canonical}

(1) For any $t\in\Delta^{+}$, there is a corresponding reachable
chamber 
\begin{align*}
\frC^{t} & =\cap_{k\in I_{\ufv}}H_{c_{k}(t)}^{\geq0}=\sum_{i\in I}\R_{\geq0}\deg x_{i}(t)+\sum_{j\in I_{\fv}}\R_{\geq0}(-f_{j}),
\end{align*}
for non-zero primitive vectors $c_{k}(t)\in\pm\N^{I_{\ufv}}$, called
the $c$-vectors. In particular, $\frC^{t_{0}}=\frC^{+}$. 

Moreover, the walls of $\frC^{t}$ takes the form $(\frd_{k},\frp_{\frd_{k}})$,
$k\in I_{\ufv}$, where 
\begin{align*}
\frd_{k} & =\sum_{i\neq k}\R_{\geq0}\deg x_{i}(t)+\sum_{j\in I_{\fv}}\R_{\geq0}(-f_{j})
\end{align*}
 has the primitive normal vector $c_{k}(t)$, and $\frp_{\frd_{k}}=\exp(-d_{k}\Li_{2}(-y^{|c_{k}(t)|}))$.

(2) Conversely, all reachable chambers take this form.

\end{Thm}

\subsection{Broken lines and theta functions}

Take $\kk=\Z$. Assume that $\tB$ is of full rank from now on. Recall
that we have the Laurent polynomial ring $\LP=\kk[x_{i}^{\pm}]_{i\in I}$
and its completion $\hLP=\kk\llbracket x_{i}^{\pm}\rrbracket_{i\in I}$.
Let $A=\kk[\yCone^{\geq0}]$ acts on $\hLP$ via the derivation $\{A,\ \}$
such that $\{y^{n},x^{m}\}:=\langle m,n\rangle x^{m+p^{*}n}$. In
particular, we have $\{y^{n},x^{p^{*}n'}\}=-\{n,n'\}x^{p^{*}(n+n')}$.
Since $p^{*}$ is injective, the action of $A$ is faithful. It induces
a faithful action of $\hG$ on $\hLP$. Therefore, we can view $\hG$
as a subgroup of $\Aut(\hLP)$.

Let $n_{0}$ denote a primitive vector in $\yCone^{>0}$. For any
given a formal series $f=1+\sum_{k>0}c_{k}y^{k\cdot n_{0}}\in\hG_{n_{0}}^{\bigparallel}\subset\hG\subset\hA$,
we define the following element in $\Aut(\hLP)$:
\begin{align*}
\frp_{f}: & \hLP\ra\hLP\\
 & x^{m}\mapsto x^{m}\cdot f^{\langle m,n_{0}\rangle}
\end{align*}

\begin{Lem}\cite[Definition 1.2, Lemma 1.3]{gross2018canonical}

For any primitive $n_{0}\in\yCone^{>0}$, $\hG_{n_{0}}^{\|}\subset\Aut(\hB)$
is the subgroup of automorphisms of the form $\frp_{f}$ with the
given $n_{0}$. 

\end{Lem}

\begin{Lem}\cite{gross2018canonical}

We have $\exp(-d_{k}\Li_{2}(-y_{k}))=\frp_{1+y_{k}}$.

\end{Lem}

\subsubsection*{Broken lines}

Given $m\in\cone$ and $Q\in M_{\R}\backslash\frD$. A broken line
$\gamma$ in the direction $m$ with base point $Q$ is a piecewise
linear map $\gamma:(-\infty,0]\rightarrow M_{\R}$ with Laurent monomials
attached to its linear components, satisfying the following properties.

\begin{itemize}

\item $\gamma(0)=Q$,

\item $\gamma$ crosses the walls of $\frD$ transversally. Moreover,
at the times $t_{1}<t_{2}<\cdots<t_{r}$, $r\in\N$, $\gamma(t_{i})$
is a general point of a wall and $L^{(i)}:=\gamma(t_{i},t_{i+1}]$,
$0\leq i\leq r$, is linear, where we denote $t_{0}=-\infty$ and
$t_{r+1}=0$.

\item the derivative $\gamma'$ on $L^{(i)}$ equals $-m^{(i)}$,
where the monomial attached to $L^{(i)}$ is denoted by $\Mono(L^{(i)})=c^{(i)}x^{m^{(i)}}$.

\item $c^{(0)}x^{m^{(0)}}=x^{m}$

\item $c^{(i+1)}x^{m^{(i+1)}}$ is a term in the formal power series
$\frp_{i}(c^{(i)}x^{m^{(i)}})$, where $\frp_{i}:=\prod_{\gamma(t_{i})\in\frd}\frp_{\frd}^{\mathrm{sign}\langle m^{(i)},n_{\frd}\rangle}$,
i.e., $\frp_{i}$ is the limit of $\frp_{\gamma(t_{i}+\epsilon,t_{i}-\epsilon)}$
as $\epsilon\rightarrow0^{+}$.

\end{itemize}

We define $\Mono(\gamma)=\Mono(L^{(r+1)})$ and $I(\gamma)=x^{m}$.
Note that $\Mono(L^{(i)})$ always take the form $c^{(i)}x^{m+p^{*}n^{(i)}}$
for some $n^{(i)}\in\yCone^{\geq0}$, $c^{(i)}\in\kk$, and $n^{(i+1)}\geq n^{(i)}$,
see \cite[Remark 3.2]{gross2018canonical}.

\subsubsection*{Theta functions}

For any general base point $Q\in M_{\R}\backslash\supp\frD$ and any
$m\in\cone$, the theta function $\vartheta_{Q,m}$ is defined to
be

\begin{align*}
\theta_{Q,m} & =\sum_{\gamma}\Mono(\gamma)
\end{align*}
where the sum runs over all broken lines $\gamma$ with the base point
$Q$ and $I(\gamma)=x^{m}$.

\begin{Thm}\cite[Theorem 3.5]{gross2018canonical}\cite[Section 4]{carl2010tropical}\label{thm:base_change_theta_function}

Let there be given two general base points $Q,Q'\in M_{\R}\backslash\supp\frD$
such that the coordinates of $Q$ (resp. of $Q'$) are linearly independent
over $\Q.$Then, for any generic smooth path $\gamma$ from $Q$ to
$Q'$, and any $m\in\cone$, we have

\begin{align*}
\theta_{Q',m} & =\frp_{\gamma}\theta_{Q,m}.
\end{align*}

\end{Thm}

For any generic points $Q,Q'$ in $\Int(\frC)$ for some chamber $\frC$,
we have $\theta_{Q,m}=\theta_{Q',m}$ by Theorem \ref{thm:base_change_theta_function}.
Correspondingly, we denote $\theta_{\frC,m}=\theta_{Q,m}$.

\begin{Prop}\cite[Proposition 3.8]{gross2018canonical}

$\theta_{\frC^{+},m}=x^{m}$ when $m\geq0$.

\end{Prop}

Let us write $\theta_{m}=\theta_{\frC^{+},m}$ for simplicity.

\begin{Thm}\cite{gross2018canonical}

For any seeds $t_{0}=\seq t$ and $m\in\cone(t_{0})$, we have $\deg^{t}\seq^{*}\theta_{m}=\phi_{t,t_{0}}m$.

\end{Thm}

\section{Basics of the freezing operators\label{sec:Basics-of-the-freezing}}

We review the basics of freezing operators used in \cite{qin2023analogs}.

\subsection{Definition}

Take any initial seed $t_{0}=((b_{ij}),(x_{i})_{i\in I},I,I_{\ufv},(d_{i}))$.
Choose and fix a subset of unfrozen vertices $F\subset I_{\ufv}$
and denote $R=I_{\ufv}\backslash F$. By freezing the vertices in
$F$, we obtain a new seed $\frz_{F}t_{0}=((b_{ij}),(x_{i})_{i\in I},I,I_{\ufv}\backslash F,(d_{i}))$.
Denote $\frz_{F}\tB:=\tB(\frz_{F}t_{0})=(b_{ij})_{(i,j)\in I\times R}$
. If $t_{0}$ is a quantum seed endowed with a compatible Poisson
structure $\lambda$, we make $\frz_{F}t_{0}$ into a quantum seed
by endow it with the same $\lambda$.

Note that $\cone(t_{0})=\cone(\frz_{F}t_{0})$. In addition, if $m'\preceq_{\frz_{F}t_{0}}m$,
then $m'\preceq_{t_{0}}m$.

\begin{Def}[Freezing operator {\cite{qin2023analogs}}]\label{def:freezing_operator}

For any $m\in\Z^{I}$ and $F\subset I_{\ufv}$, the corresponding
freezing operator is defined to be the $\kk$-linear map $\frz_{F,m}^{t_{0}}$
from $\hLP_{\preceq m}(t_{0})$ to $\hLP_{\preceq m}(\frz_{F}t_{0})$
such that 
\begin{align*}
\mathfrak{\frz}_{F,m}^{t_{0}}(x^{m+p^{*}n}) & =\begin{cases}
x^{m+p^{*}n} & \supp n\cap F=\emptyset\\
0 & \supp n\cap F\neq\emptyset
\end{cases}
\end{align*}

\end{Def}

Consider the $\kk$-vector space $\kk[\N^{I_{\ufv}}]:=\oplus_{n\in\N^{I_{\ufv}}}\kk y^{n}$.
It is isomorphic to $\kk[M^{\oplus}]$ as a $\kk$-vector space such
that $y^{n}$ is sent to $x^{p^{*}n}$. We can then endow it with
the $v$-twisted multiplication such that it is isomorphic to $\kk[M^{\oplus}]\subset\LP$.
The completion $\widehat{\kk[\N^{I_{\ufv}}]}$ is naturally defined,
such that it is isomorphic to $\widehat{\kk[M^{\oplus}]}$. It would
be convenient to introduce the following definition.

\begin{Def}[Freezing operator {\cite{qin2023analogs}}]\label{def:freezing_operator-1}

The freezing operator from $\widehat{\kk[\N^{I_{\ufv}}]}$ to $\widehat{\kk[\N^{I_{\ufv}\backslash F}]}$
is the $\kk$-algebra homomorphism such that 
\begin{align*}
\mathfrak{\frz}_{F}(y^{n}) & =\begin{cases}
y^{n} & \supp n\cap F=\emptyset\\
0 & \supp n\cap F\neq\emptyset
\end{cases}
\end{align*}

\end{Def}

We often drop the symbols $t_{0}$, $F$, $m$ for convenience. It
is straightforward to check the following results.

\begin{Lem}[{\cite{qin2023analogs}}]\label{lem:projection_properties}

(1) $\frz_{F,m}$ is surjective.

(2) Take any $m'\preceq_{t_{0}}m$ such that $m'=m+\tB n$. $\forall z\in\hLP_{\preceq m'}(t_{0})\subset\hLP_{\preceq m}(t_{0})$,
we have 
\begin{align*}
\frz_{m}(z) & =\begin{cases}
\frz_{m'}(z) & \supp n\cap F=\emptyset\\
0 & \supp n\cap F\neq\emptyset
\end{cases}
\end{align*}

(3) For any $z_{1}\in\hLP_{\preceq m_{1}}(t_{0})$, $z_{2}\in\hLP_{\preceq m_{2}}(t_{0})$,
we have
\begin{align*}
\frz_{m_{1}}(z_{1})*\frz_{m_{2}}(z_{2}) & =\frz_{m_{1}+m_{2}}(z_{1}*z_{2}).
\end{align*}

\end{Lem}

\begin{Prop}[{\cite{qin2023analogs}}]\label{prop:project_up_cl_alg}

Let there be given a pointed element $z\in\LP(t_{0})$. If $z\in\upClAlg(t_{0})$,
then $\frz_{F,\deg z}^{t_{0}}z\in\upClAlg(\frz_{F}t_{0})$.

\end{Prop}

\begin{proof}

Let $\seq$ be any sequence of mutations on $I_{\ufv}\backslash F$.
Denote $t=\seq t_{0}$. Then we have $\frz_{F}t=\seq\frz_{F}t_{0}$.
Since $z$ is a Laurent polynomial in the cluster variables of $t$,
we can write $z*x(t)^{d}=\sum c_{m}x(t)^{m}$, $c_{m}\in\kk$, $d,m\in\N^{I}$.
View this equality in $\LP(t_{0})$ and applying the freezing operator
$\frz_{F,\deg z+\deg x(t)^{d}}^{t_{0}}$ to both sides. Note that
$\frz_{F,\deg x_{i}(t)}^{t_{0}}x_{i}(t)=x_{i}(t)=x_{i}(\frz_{F}t)$.
So Lemma \ref{lem:projection_properties}(3)(2) implies that $\frz_{F,\deg z}^{t_{0}}(z)$
is still a Laurent polynomial in $x_{i}(\frz_{F}t)$. The desired
claim follows.

\end{proof}

\begin{Eg}

Take $\kk=\Z$. We define the seed $t_{0}$ such that $I=\{1,2\}$,
$I_{\ufv}=\{1\}$, $(b_{ij})_{i,j\in I}=\left(\begin{array}{cc}
0 & -1\\
1 & 0
\end{array}\right)$, $d_{1}=d_{2}=1$. The upper cluster algebra $\upClAlg(t_{0})$ equals
$\kk[x_{2}^{\pm}][x_{1},x_{1}']$ where $x_{1}':=\frac{1+x_{2}}{x_{1}}=x_{1}^{-1}\cdot(1+x^{p^{*}e_{1}})$.
The localized cluster monomials of $\upClAlg(t_{0})$ take the form
$s_{m}:=\begin{cases}
x_{1}^{m_{1}}x_{2}^{m_{2}} & m_{1}\geq0\\
(x_{1}')^{-m_{1}}x_{2}^{m_{2}} & m_{1}<0
\end{cases}$, where $m=(m_{1},m_{2})\in\Z^{2}$, and they form a basis $\cS$
for $\upClAlg(t_{0})$.

Take $F=I_{\ufv}=\{1\}$ and construct the seed $\frz_{F}t_{0}$ by
freezing $F$. We have $\upClAlg(\frz_{F}t_{0})=\kk[x_{1}^{\pm},x_{2}^{\pm}]$.
Its localized cluster monomials $s'_{m}:=x^{m}$ form a basis. For
each $m$-pointed (localized) cluster monomial $s_{m}$ in $\upClAlg(t_{0})$,
the freezing operator $\frz_{F,m}$ sends it to $s'_{m}$ in $\upClAlg(\frz_{F}t_{0})$.
We can check $\frz_{F,m+m'}(s_{m}\cdot s_{m'})=\frz_{F,m}(s_{m})\cdot\frz_{F,m'}(s_{m'})$
as in Lemma \ref{lem:projection_properties}(3).

Correspondingly, we can construct a bijective $\kk$-linear map $\frz_{F}^{\cS}:\upClAlg(t_{0})\rightarrow\upClAlg(\frz_{F}t_{0})$,
sending $s_{m}$ to $s'_{m}$. As seen in Lemma \ref{lem:projection_properties}(2),
this linear map $\frz_{F}^{\cS}$ does not preserve the multiplication.
For example, we have $\frz_{F}^{\cS}x_{1}=x_{1}$, $\frz_{F}^{\cS}x_{1}'=x_{1}^{-1}$,
but $\frz_{F}^{\cS}(x_{1}\cdot x_{1}')=\frz_{F}^{\cS}(1+x_{2}):=\frz_{F}^{\cS}(1)+\frz_{F}^{\cS}(x_{2})=1+x_{2}\neq\frz_{F}^{\cS}(x_{1})\cdot\frz_{F}^{\cS}(x_{1}')$.

\end{Eg}

\subsection{Properties}

\begin{Lem}[{\cite{qin2023analogs}}]\label{lem:freeze_mutation}

Assume that an element $z\in\upClAlg(t_{0})$ has a unique $\prec_{t}$-maximal
degree in each $\LP(t)$, $t\in\Delta_{t_{0}}^{+}$. Then, for any
mutation sequence $\seq$ on $I_{\ufv}\backslash F$, we have $\seq^{*}\frz_{F,\deg^{t_{0}}z}^{t_{0}}z=\frz_{F,\deg^{t}z}^{t}\seq^{*}z$
where $t_{0}=\seq t$.

\end{Lem}

\begin{proof}

We give a proof for the convenience of the reader.

Denote $z=\sum_{n\in\N^{I_{\ufv}}}c_{n}x^{m+p^{*}n}\in\LP(t_{0})$,
$c_{n}\in\kk$, $c_{0}\neq0$. By \cite[Section 3.5]{qin2019bases},
we have $\seq^{*}z=\sum_{n}c_{n}\seq^{*}(x^{m+p^{*}n})$ in $\hLP(t)$,
where $\seq^{*}(x^{m+p^{*}n})\in\hLP(t)$ on the right hand side is
the formal Laurent series of the rational function $\seq^{*}(x^{m+p^{*}n})\in\cF(t)$.
Let us denote $N_{1}=\{n|c_{n}\neq0,n\in m+p^{*}\N^{I_{\ufv}\backslash F}\}$
and $N_{1}=\{n|c_{n}\neq0,n\notin m+p^{*}\N^{I_{\ufv}\backslash F}\}$.
Then $\frz_{F}^{t_{0}}z=\sum_{n\in N_{1}}c_{n}x^{m+p^{*}n}$. 

By assumption, $\seq^{*}z$ has a unique $\prec_{t}$-maximal degree
in $\LP(t)$, which must be the contribution of a unique $x^{m+p^{*}n'}$
for some $n'$. Similarly, denote $N_{1}(t)=\{n|c_{n}\neq0,\deg^{t}\seq^{*}x^{m+p^{*}n}\in\deg^{t}\seq^{*}x^{m+p^{*}n'}+p^{*}\N^{I_{\ufv}\backslash F}\}$
and $N_{2}(t)=\{n|c_{n}\neq0,\deg^{t}\seq^{*}x^{m+p^{*}n}\notin\deg^{t}\seq^{*}x^{m+p^{*}n'}+p^{*}\N^{I_{\ufv}\backslash F}\}$.
Then we also have $\frz_{F}^{t}\seq^{*}z=\sum_{n\in N_{1}(t)}c_{n}\seq^{*}x^{m+p^{*}n}$.

By the mutation rules of $x^{p^{*}n}$ (see (\ref{eq:mutation-y})),
the factors $x^{p^{*}e_{j}}$, $j\in F$, in the leading term are
never created or eliminated by $\seq^{*}$. Since $\deg^{t}\seq^{*}x^{m+p^{*}n'}\succeq_{t}\deg^{t}\seq^{*}x^{m}$,
the $j$-th coordinate $n'_{j}$ of $n'$ must vanish for all $j\in F$.
Using the mutation rules of $x^{p^{*}n}$ again, we obtain that $n\in N_{1}(t)$
if and only if $n\in N_{1}$. Therefore $\frz_{F}^{t}\seq^{*}z=\seq^{*}\frz_{F}^{t_{0}}z$.

\end{proof}

Lemma \ref{lem:projection_properties} immediately implies the following
result.

\begin{Lem}[{\cite{qin2023analogs}}]\label{lem:freeze_compatibly_pointed}

If an element $z\in\upClAlg(t_{0})$ is compatibly pointed at all
seeds in $\Delta_{t_{0}}^{+},$ then $\frz_{F}z\in\upClAlg(\frz_{F}t_{0})$
is compatibly pointed at all seeds in $\Delta_{\frz_{F}t_{0}}^{+}$.

\end{Lem}

\begin{Thm}[{\cite{qin2023analogs}}]\label{thm:freeze_tropical_bases}

Assume that $\upClAlg(t_{0})$ possesses a $\Z^{I}$-pointed subset
$\cS=\{s_{g}|g\in\Z^{I}\}$, such that $\deg^{t_{0}}s_{g}=g$ and
$s_{g}$ are compatibly pointed at all seeds in $\Delta_{t_{0}}^{+}$.
If $\frz_{F}t_{0}$ is injective-reachable, then $\frz_{F}\cS:=\{\frz_{F,g}s_{g}|g\in\Z^{I}\}$
is a basis of $\upClAlg(\frz_{F}t_{0})$.

\end{Thm}

\begin{proof}

By Lemma \ref{lem:freeze_compatibly_pointed}, the elements of $\frz_{F}\cS$
are compatibly pointed in $\Delta_{\frz_{F}t_{0}}^{+}$. Since $\frz_{F}t_{0}$
is injective-reachable, by \cite[Theorem 4.3.1]{qin2019bases}, $\frz_{F}\cS$
is a basis for $\upClAlg(\frz_{F}t_{0})$.

\end{proof}

\section{Freezing operators on scattering diagrams\label{sec:Projection_of_scattering}}

Let there be given an initial seed $t_{0}$ such that $\tB(t_{0})$
is of full rank. In the following, we will work at either the classical
level for any $t_{0}$, or at the quantum level but assume that $B(t_{0})$
is skew-symmetric. Then we have the corresponding consistent scattering
diagram $\frD(t_{0})$, such that the localized cluster monomials
are theta functions, and the diagram can be mutated, see \cite{gross2018canonical}\cite{davison2019strong}. 

We refer the reader to Section \ref{sec:Scattering-diagram} for a
detailed review of the classical scattering diagrams. Quantum scattering
diagrams for skew-symmetric $B(t_{0})$ have similar descriptions,
and could be found in \cite[Section 5]{MandelQin2021}.

\subsection{Freezing scattering diagrams\label{subsec:Freezing_scattering_diagrams}}

Let there be given an initial seed $t_{0}$. Let $\frD=\frD(t_{0})$
denote the cluster scattering diagram associated to the seed $t_{0}$.
Recall that $\frD$ is a collection of walls $(\frd,\frp_{\frd})$,
such that $\frd$ is a codimension $1$ cone in $\cone_{\R}=\R^{I}$
with a primitive normal vector $n_{0}\in\yCone^{>0}$, and the wall
crossing operator $\frp_{\frd}$. We have broken lines, which are
piecewise linear curves in $\cone_{\R}$ bending at walls, coming
from a direction $m$, and ending at a chosen base point $\sQ$. 

Choose a subset of unfrozen vertices $F\subset I_{\ufv}$ and denote
$R=I_{\ufv}\backslash F$. Then $\frg_{R}:=\oplus_{\supp n\subset R}\frg_{n}$
is a Lie subalgebra of $\frg$. Moreover, we have $\frg_{R}=\frg/\mathfrak{j}$,
where $\mathfrak{j}:=\oplus_{n:\supp n\cap F\ne\emptyset}\kk y^{n}$
is an ideal of the Lie algebra $\frg$. Let $\frz_{F}$ denote the
natural projection $\frg\rightarrow\frg_{R}$. Then $\frz_{F}$ is
a Lie algebra homomorphism.

The homomorphism $\frz_{F}$ induces a homomorphism between the completion
$\hfrg\rightarrow\hfrg_{R}$ and, consequently, a homomorphism $\frz_{F}$
between the corresponding pro-nilpotent group $\hG\rightarrow\hG_{R}$.
Then we obtain a scattering diagram $\frz_{F*}\frD$ in $\cone_{\R}$
via the pushforward $\frz_{F*}$, which consists of the walls $(\frd,\frz_{F}\frp_{\frd})$
for any wall $(\frd,\frp_{\frd})$ in $\frD$, see Lemma \ref{lem:exp_intertwine_morphism}.
Note that, if $\frD$ is consistent, then its pushforward $\frz_{F*}\frD$
is consistent as well. We view $\frz_{F*}\frD$ as the scattering
diagram constructed from $\frD$ by freezing the vertices in $F$.

Let $\frz_{F}t_{0}$ denote the seed obtained from $t_{0}$ by freezing
the vertices in $F$. Let $\frD(t_{0})$ and $\frD(\frz_{F}t_{0})$
denote the cluster scattering diagrams associated to $t_{0}$ and
$\frz_{F}t_{0}$ respectively. Then we have the following result.

\begin{Lem}\label{lem:pushforward_scattering_diagram}

The scattering diagram $\frz_{F*}\frD(t_{0})$ is equivalent to $\frD(\frz_{F}t_{0})$. 

\end{Lem}

\begin{proof}

the nontrivial incoming walls of $\frz_{F*}\frD(t_{0})$ takes the
form $(e_{k}^{\bot},\exp(-d_{k}\Li_{2}(-y_{k})),k\in R$, in the classical
case, and the quantum analog in the quantum case (see \cite[(59) (71)]{MandelQin2021}).
Therefore $\frz_{F*}\frD(t_{0})$ and $\frD(\frz_{F}t_{0})$ have
the same set of nontrivial incoming walls. The claims follows from
the fact that they are both consistent.

\end{proof}

\begin{Thm}\label{thm:projection_theta_function}

For any theta function $\theta_{Q,g}$ of $\frD(t_{0})$, $g\in\cone(t)$,
$Q$ a generic base point, $\frz_{F,g}\theta_{Q,g}$ is the theta
function $\theta_{Q,g}^{'}$ of $\frD(t')$ where $t'=\frz_{F}t$.

\end{Thm}

\begin{proof}

For any broken line $L$ in $\frD(t)$ with initial direction $g$,
$\frz_{F}\Mono(L)\neq0$ if and only if $L$ only bends at nontrivial
walls $\frd\subset n^{\bot}$ such that $\supp n\subset R$, which
are walls that also belong to $\frD(t')$. Therefore, $L$ is also
a broken line in $\frD(t')$.

Conversely, any broken line $L$ in $\frD(t')$ is also a broken line
in $\frD(t)$. The claim follows.

\end{proof}

\subsection{Chamber structures and cluster monomials}

Recall that the walls in $\frD(t_{0})$ cut off chambers from $\cone_{\R}\simeq\R^{I}$,
such that the seeds $t\in\Delta^{+}$ correspond to chambers $\cC^{t}$.
In particular, $\cC^{t_{0}}$ is the positive quadrant $\cC^{+}:=\{m|m_{k}\geq0|\forall k\in I_{\ufv}\}$.
Moreover, two seeds related by a single step mutation correspond to
two chambers sharing a common wall. 

\begin{Lem}\label{lem:chamber_merge}

Any chamber of $\frD(t_{0})$ is contained in a chamber of $\frD(\frz_{F}t_{0})$.

\end{Lem}

\begin{proof}

There are fewer walls in $\frD(\frz_{F}t_{0})$ than in $\frD(t_{0})$.
Therefore, the chambers of $\frD(t_{0})$ are contained in chambers
of $\frD(\frz_{F}t_{0})$.

\end{proof}

Recall that a chamber $\cC$ is said to be reachable (from $\cC^{+}$)
in $\frD(t_{0})$ if there is a generic path connecting the interiors
$\Int(\cC^{+})$, $\Int(\cC)$ that intersects only finitely many
walls in $\frD(t_{0})$, see Definition \ref{def:reachable_chamber}.

\begin{Lem}\label{lem:theta_localized_cluster_monomial}

A theta function $\theta_{g}$ is a localized cluster monomial of
a cluster algebra $\upClAlg(t_{0})$ if and only if $g\in\cC$ for
some chamber $\cC$ of $\frD(t_{0})$ reachable from $\cC^{+}$.

\end{Lem}

\begin{proof}

The claim follows from the mutation of scattering diagrams, see \cite{gross2018canonical}.

\end{proof}

\begin{Prop}\label{prop:projection_localized_cluster_monomial}

We work with $\kk=\Z$, or with $\kk=\Z[v^{\pm}]$ but assuming $B(t_{0})$
is skew-symmetric. If $z$ is a localized cluster monomial in $\upClAlg(t_{0})$,
then $\frz_{F,\deg z}z$ is a localized cluster monomial in $\clAlg(\frz_{F}t_{0})$.

\end{Prop}

\begin{proof}

Since $z$ is a localized cluster monomial, it must agree with the
$\deg z$-pointed theta function $\vartheta_{\deg z}$ and $\deg z$
is contained in some reachable chamber $\cC$ in $\frD(t_{0})$, see
Lemma \ref{lem:theta_localized_cluster_monomial}. By Lemma \ref{lem:chamber_merge},
$\deg z$ must be contained in some reachable chamber $\cC'$ in $\frD(\frz_{F}t_{0})$
that contains $\cC$. The claim follows.

\end{proof}

Let us discuss the mutation reachability of chambers. We start with
the following observation by Muller \cite{muller2015existence}.

\begin{Lem}\label{lem:projection_reachable_chamber}

Let $\frC$ and $\frC'$ denote chambers in the scattering diagram
$\frD(t_{0})$ and $\frz_{F}\frD(t_{0})$ respectively, such that
$\frC\subset\frC'$. If $\frC$ is reachable in $\frD(t_{0})$, then
$\frC'$ is reachable in $\frz_{F}\frD(t_{0})$ as well.

\end{Lem}

\begin{proof}

Since $\frC$ is reachable, there exists a $\frD(t_{0})$-generic
smooth path $\gamma$ from $\Int(\cC^{+})$ to $\Int(\frC)$ intersecting
finitely many walls in $\frD(t_{0})$. By the construction of $\frz_{F}\frD(t_{0})$,
$\gamma$ is also $\frz_{F}\frD(t_{0})$-generic and intersects fewer
walls in $\frz_{F}\frD(t_{0})$.

\end{proof}

Note that $\cC^{t_{0}[1]}=\cC^{-}:=-\cC^{+}$ when $t_{0}$ is injective-reachable.
So Lemma \ref{lem:projection_reachable_chamber} implies the following
result.

\begin{Thm}[{\cite[Theorem 1.4.1]{muller2015existence}}]\label{thm:freeze_inj_reachable}

If $t_{0}$ is injective-reachable, then $\frz_{F}t_{0}$ is injective-reachable
too.

\end{Thm}

We then deduce the following result from the above observation (Lemma
\ref{lem:projection_reachable_chamber}). \cite{cao2020enough} also
gave a proof for this result based on Muller's observation.

\begin{Thm}\label{thm:seed_rechability_reduction}

Given a seed $t\in\Delta_{t_{0}}^{+}$ such that $x_{j}(t_{0})$ are
cluster variables of $t$ for any $j\in F$. Then we can find a mutation
sequence $\seq=\mu_{k_{r}}\cdots\mu_{k_{1}}$ such that $t=\seq t_{0}$
and $k_{s}\notin F$ for any $1\leq s\leq r$. In particular, $\supp x_{i}(t)\subset R$
for any $i$.

\end{Thm}

\begin{proof}

Recall that the seeds $t_{0}$ and $t$ correspond to chambers $\frC^{t_{0}}$
and $\frC^{t}$ of $\frD(t_{0})$. Let us consider the scattering
diagram $\frD(t_{0}')$, where $t_{0}'=\frz_{F}t_{0}$. We have an
inclusion $\frC^{t_{0}}\subset\frC^{t_{0}'}$ and $\frC^{t}\subset\frC'$
for chambers $\frC^{\frz_{F}t_{0}}$ and $\frC'$ of $\frD(\frz_{F}t_{0})$.
Note that $\frC^{'}$ is reachable from $\frC^{\frz_{F}t_{0}}$ in
$\frD(\frz_{F}t_{0})$ by Lemma \ref{lem:projection_reachable_chamber}.
Therefore, there exists a $t'=\seq t'_{0}$ for some mutation sequence
$\seq$ on the set of vertices $R$ such that $\frC'=\frC^{t'}$.
For showing $t=\seq t_{0}$, it suffices to verify $\frC^{t}=\frC^{\seq t_{0}}$.

Note that we have $\frC^{t_{0}}+\Z^{F}=\frC^{t_{0}'}$. Applying the
sequence of tropical mutations $\phi_{t_{0}',t'}^{-1}=\phi_{t_{0}',\seq t'_{0}}^{-1}=\phi_{t_{0},\seq t_{0}}^{-1}$,
these chambers are mapped to $\frC^{\seq t_{0}}+\Z^{F}=\frC^{t'}$.
We deduce that $\frC^{t}\subset\frC^{\seq t_{0}}+\Z^{F}$. It follows
that the intersection between the interiors of $\frC^{t}$ and $\frC^{\seq t_{0}}$
is non-empty. Therefore, $\frC^{t}=\frC^{\seq t_{0}}$.

\end{proof}

\begin{Cor}

Let there be given a cluster monomial $z$ in a seed $t\in\Delta_{t_{0}}^{+}$
such that $\supp z\cap F=\emptyset$. Then $z$ is a cluster monomial
in $t'$ for some $t'\in\Delta_{\frz_{F}t_{0}}^{+}$. 

\end{Cor}

\subsection{Freezing of cluster monomials}

Take any mutation sequence $\seq=\seq_{\uk}$ and denote a seed $t=\seq t_{0}$.
Let $u=x(t)^{m}\in\LP(t)$ denote a localized cluster monomial of
$t$, $m\in\N^{I_{\ufv}}\oplus\Z^{F}$. Recall that we have $\seq^{*}u=x^{g}\cdot F_{u}|_{y^{n}\mapsto x^{p^{*}n}}$,
where $F_{u}^{t_{0}}=\sum_{n\in\N^{I_{\ufv}}}c_{n}y^{n}$ is the $F$-polynomial
with $c_{0}=1$. We define 
\begin{align*}
\frz_{F}^{t_{0}}\seq^{*}u & :=x^{g}\cdot(\frz_{F}F_{u}^{t_{0}})|_{y^{n}\mapsto x^{p^{*}n}}=x^{g}\cdot\sum_{n:\supp n\cap F=\emptyset}c_{n}x^{p^{*}n}.
\end{align*}

Note that the above $\frz_{F}^{t_{0}}\seq^{*}u$ is well-defined even
when $\tB(t_{0})$ is not of full rank. When $\tB(t_{0})$ is of full
rank, the above $\frz_{F}^{t_{0}}\seq^{*}u$ equals $\frz_{F,g}^{t_{0}}\seq^{*}u$
in Definition \ref{def:freezing_operator}.

Let $N(F_{u}^{t_{0}})$ denote the Newton polytope of $F_{u}^{t_{0}}$,
i.e., the convex hull of $\{n\in\Z^{I_{\ufv}}|c_{n}\neq0\}$. For
the quantum case $\kk=\Z[v^{\pm}]$, we denote it by $N_{v}(F_{u}^{t_{0}})$.
For the classical case, we denote it by $N_{1}(F_{u}^{t_{0}})$. 

\begin{Lem}[{\cite[Section 5]{Tran09}}]\label{lem:newton_polytope_F_poly}

We have $N_{v}(F_{u}^{t_{0}})=N_{1}(F_{u}^{t_{0}})$.

\end{Lem}

\begin{Thm}

If $z$ is a localized (quantum) cluster monomial of $\upClAlg(t_{0})$,
then $\frz_{F}^{t_{0}}z$ is a localized (quantum) cluster monomial
of $\upClAlg(\frz_{F}t_{0})$.\label{thm:projection_quantum_cluster_monomial}

\end{Thm}

\begin{proof}

Denote $z'=\frz_{F}^{t_{0}}z$, $t_{0}'=\frz_{F}t_{0}$.

Assume $\kk=\Z$ for the moment. In this case, the desired statement
has been proved in Proposition \ref{prop:projection_localized_cluster_monomial}.
Then $z'$ is the Laurent expansion of a localized classical cluster
monomial. So there exists a seed $s'\in\Delta_{t_{0}'}^{+}$ and a
localized classical cluster monomial of $u'=x(s')^{m}$, $m\in\N^{I_{\ufv}\backslash F}\oplus\Z^{I_{\fv}\sqcup F}$,
such that $z'$ is the Laurent expansion of $u'$ in $\LP(t_{0}')$.

Note that we have $t_{0}'=\seq s'$ for some sequence of mutation
$\seq$ on the letters from $I_{\ufv}\backslash F$. Then $\seq^{*}(z')$
equals $u'$. Therefore, in $\LP(s')$, $\seq^{*}(z')$ is $m$-pointed
and its $F$-polynomial $F_{z'}^{s'}$ is $1$.

Denote $s=(\seq)^{-1}t_{0}$. Recall that the localized cluster monomial
$z$ is compatibly pointed at all seeds. Then Lemmas \ref{lem:freeze_mutation}
and \ref{lem:freeze_compatibly_pointed} imply $\frz_{F}^{s}(\seq^{*}z)=\seq^{*}(\frz_{F}^{t_{0}}z)=\seq^{*}z'\in\LP(s')$.
Let $F_{z}^{s}=\sum_{n\in\N^{I_{\ufv}}}c_{n}(1)y^{n}$, $c_{n}(1)\in\Z$,
denote the classical $F$-polynomial of $\seq^{*}z$ in the seed $s$.
Then, we have $\frz_{F}^{s}F_{z}^{s}=F_{z'}^{s'}=1$ or, equivalently,
$\{n|c_{n}(1)\neq0,\supp n\cap F=\emptyset\}=\{0\}$. We deduce that
the Newton polytope $N_{1}$ of $F_{z}^{s}$ satisfies $N_{1}\cap\N^{I_{\ufv}}=\{0\}$,
where we identify $\N^{I_{\ufv}}$ with $\N^{I_{\ufv}}\oplus0\subset\N^{I}$.

Now work at the quantum case $\kk=\Z[v^{\pm}]$. Note that $\seq^{*}z$
is still $m$-pointed. Let $F_{z}^{s}(v)=\sum_{n\in\N^{I_{\ufv}}}c_{n}(v)\cdot y^{n}$,
denote the quantum $F$-polynomial of $\seq^{*}z$ in the seed $s$,
where $c_{n}(v)\in\Z[v^{\pm}]$ satisfy $c_{n}(v)|_{v\mapsto1}=c_{n}(1)$.
By Lemma \ref{lem:newton_polytope_F_poly}, the Newton polytope $N_{v}$
of $F_{z}^{s}(v)$ and the Newton polytope $N_{1}$ of $F_{z}^{s}$
are the same. So $N_{v}\cap\N^{I_{\ufv}}=\{0\}$. It follows that
$\{n|c_{n}(v)\neq0,\supp n\cap F=\emptyset\}=\{0\}$. Therefore, $\frz_{F}^{s}\seq^{*}z$
equals the localized quantum cluster monomial $x(s')^{m}$ of the
seed $s':=(\seq)^{-1}t_{0}'$ as desired.

\end{proof}

\section{Bases constructions: freezing and localization }

\subsection{Localness\label{subsec:Localness}}

Let there be given a subset $\Theta\subset\cone(t_{0})$ and $\cS=\{s_{g}|g\in\Theta\}$
a collection of pointed formal Laurent series in $\hLP(t_{0})$ such
that $\deg s_{g}=g$. Note that $s_{g}$ are linearly independent
in $\hLP(t_{0})$. We define the $\kk$-module $\clAlg^{\Theta}:=\oplus_{g\in\Theta}\kk s_{g}$.

\begin{Def}

A formal sum $\sum b_{g}s_{g}$, $b_{g}\in\kk$, is said to be $\prec_{t}$-unitriangular,
if $\{g|b_{g}\neq0\}$ has a unique $\prec_{t}$-maximal element $g_{0}$
and, moreover, $b_{g_{0}}=1$. If further $b_{g}\in\mm:=v^{-1}\Z[v^{-1}]$
whenever $g\neq g_{0}$, the sum is said to be $(\prec_{t},\mm)$-unitriangular.

\end{Def}

Note that a $\prec_{t}$-unitriangular formal sum is a well defined
sum in $\widehat{\LP}(t_{0})$.

\begin{Lem}[{\cite[Lemma 3.1.10]{qin2017triangular}}]\label{lem:finit_sum_triangular}

If $z$ is pointed and $z=\sum b_{g}s_{g}$, $b_{g}\in\kk$, is a
finite sum in $\hLP(t_{0})$, then the sum is a $\prec_{t}$-unitriangular.

\end{Lem}

\begin{Def}[Local support under multiplication]\label{def:local_support}

A collection of pointed Laurent series $\cS=\{s_{g}|g\in\Theta\}$
in $\hLP(t_{0})$ is said to have local support under multiplication
(or multiplicative local support, or local support), if we have a
$\prec_{t_{0}}$-unitriangular decomposition for any $g_{1},g_{2}\in\Theta$:

\begin{align}
v^{-\lambda(g_{1},g_{2})}s_{g_{1}}*s_{g_{2}} & =\sum b_{g}s_{g},\ b_{g}\in\kk,\label{eq:multiplication_decomposition}
\end{align}
such that $\supp s_{g}\subset(\supp s_{g_{1}}\cup\supp s_{g_{2}})$
whenever $b_{g}\neq0$.

Assume that $\cS$ consists of pointed Laurent polynomials. We further
say that $\cS$ has local support dimensions (under multiplication),
if $\suppDim s_{g}\leq\suppDim s_{g_{1}}+\suppDim s_{g_{2}}$, and
the inequality is strict whenever $g\neq g_{1}+g_{2}$.

\end{Def}

We have the following observation.

\begin{Rem}

If all elements in $\cS$ have non-negative Laurent coefficients,
and all $b_{g}$ in (\ref{eq:multiplication_decomposition}) are non-negative,
then $\cS$ has the the local support property. If $\cS$ is further
contained in $\LP(t_{0})$, then it has the local support dimension
property.

\end{Rem}

We give the following sufficient condition for the multiplicative
local support property.

\begin{Prop}

Assume that $\clAlg^{\Theta}$ is closed under multiplication, $\cS$
has positive multiplicative structure constants, and all $s_{g}\in\cS$
have positive Laurent expansions, then $\cS$ has local support under
multiplication. When $\clAlg^{\Theta}\subset\LP(t_{0})$, then $\cS$
has local support dimensions.

\end{Prop}

\begin{proof}

For any $s_{g}\in\Theta$, let us denote its Laurent expansion as
$s_{g}=\sum_{n'}c_{g,n'}x^{g+p^{*}n'}$, where $c_{g,n'}$ are positive
by assumption. For any $g_{1},g_{2}\in\Theta$, we have $\prec_{t_{0}}$-unitriangular
$s_{g_{1}}*s_{g_{2}}=\sum b_{g_{1},g_{2}}^{g}s_{g}$ with $b_{g_{1},g_{2}}^{g}$
being positive. Denote $g=g_{1}+g_{2}+\tB\cdot n$ for $s_{g}$ appearing.
By taking the Laurent expansion, we get
\begin{align*}
(\sum_{n_{1}}c_{g_{1},n_{1}}x^{g_{1}+p^{*}n_{1}})*(\sum_{n_{2}}c_{g_{2},n_{2}}x^{g_{2}+p^{*}n_{2}}) & =\sum_{g,n'}b_{g_{1},g_{2}}^{g}c_{g,n'}x^{g_{1}+g_{2}+p^{*}(n+n')}
\end{align*}
Equivalently, $\sum_{n_{1},n_{2}}v^{\alpha(n_{1},n_{2})}c_{g_{1},n_{1}}c_{g_{2},n_{2}}x^{p^{*}(n_{1}+n_{2})}=\sum_{g,n'}b_{g_{1},g_{2}}^{g}c_{g,n'}x^{p^{*}(n+n')}$,
where $\alpha(n_{1},n_{2}):=\lambda(g_{1}+p^{*}n_{1},g_{2}+p^{*}n_{2})$.
Because all coefficients appearing are non-negative, every Laurent
monomial $x^{p^{*}(n+n')}$ appearing on the right must be the contribution
of some $x^{p^{*}(n_{1}+n_{2})}$ on the left. In particular, we have
$\supp(n+n')\subset(\supp n_{1}\cup\supp n_{2})$. Moreover, the $k$-th
coordinates satisfy $\max\{(n+n')_{k}|b_{g_{1},g_{2}}^{g},c_{g,n'}\neq0\}\leq\max\{(n_{1}+n_{2})_{k}|c_{g_{1},n_{1}},c_{g_{2},n_{2}}\neq0\}$
for any $k\in I_{\ufv}$ when $\max\{(n_{1}+n_{2})_{k}|c_{g_{1},n_{1}},c_{g_{2},n_{2}}\neq0\}$
exists. The desired claims follow.

\end{proof}

Lemma \ref{lem:finit_sum_triangular} implies the following result.

\begin{Prop}\label{prop:same_pointed_degrees}

If $\cZ$ is a $\Theta'$-pointed subset of $\hLP(t_{0})$ for some
$\Theta'\subset\cone(t_{0})$ such that it is a $\kk$-basis for $\clAlg^{\Theta}$,
then $\Theta'=\Theta$.

\end{Prop}

Let there be given another collection of pointed Laurent series $\cZ=\{z_{g}|g\in\Theta\}$
in $\hLP(t_{0})$ such that there is a $\prec_{t_{0}}$-unitriangular
transition $(b_{g,g'})_{g,g'\in\Theta}$ from $\cS$ to $\cZ$:

\begin{align}
z_{g} & =\sum_{g'}b_{g,g'}s_{g'},\ b_{g,g'}\in\kk.\label{eq:pointed_elements_transition}
\end{align}

Taking the inverse, the transition from $\cZ$ to $\cS$ is $\prec_{t_{0}}$-unitriangular
as well, see \cite[Lemma 3.1.11]{qin2017triangular}.

\begin{Def}[Local transition]\label{def:local_transition}

The transition (\ref{eq:pointed_elements_transition}) is said to
be finite if, for any $g$, only finitely many $b_{g,g'}$ do not
vanish. 

We say the transition is local if $\supp s_{g'}\subset\supp z_{g}$
whenever $b_{g,g'}\neq0$.

If both $\cS$ and $\cZ$ are contained in $\LP(t_{0})$, the transition
if further said to be local in dimension if we have $\suppDim s_{g'}\leq\suppDim z_{g}$,
and the inequality is strict whenever $g'\neq g$.

\end{Def}

Note that if the right hand side of (\ref{eq:multiplication_decomposition})
or (\ref{eq:pointed_elements_transition}) is a finite sum, it must
be $\prec_{t_{0}}$-unitriangular by Lemma \ref{lem:finit_sum_triangular}.

\begin{Rem}

Even if there is a local transition from $\cS$ to $\cZ$, it is not
necessarily true that there is a local transition from $\cZ$ to $\cS$.
For example, one can take $\cS$ to be the dual canonical basis and
$\cZ$ the dual PBW basis (see \cite{GeissLeclercSchroeer11} for
the cluster structure on quantum groups). Then a frozen variable belongs
to $\cS$. It is a finite sum of dual PBW basis elements with larger
support.

\end{Rem}

\begin{Prop}

Assume that $\cZ$ has local support (resp. local support dimensions)
under multiplication, and the transitions from $\cS$ to $\cZ$ and
from $\cZ$ to $\cS$ are both local (resp. local in dimensions).
Then $\cS$ has local support (resp. local support dimensions) under
multiplication.

\end{Prop}

\begin{proof}

Let us denote $\prec_{t_{0}}$-unitriangular transitions $s_{g}=\sum_{g'}b_{g,g'}z_{g'}$,
$z_{g}=\sum_{g'}c_{g,g'}s_{g'}$, and $v^{-\lambda(g_{1},g_{2})}z_{g_{1}}*z_{g_{2}}=\sum_{g}\alpha_{g_{1,}g_{2}}^{g}z_{g}$.
For any $g_{1},g_{2}\in\Theta$, we have the product 
\begin{align*}
s_{g_{1}}*s_{g_{2}} & =(\sum b_{g_{1},g'}z_{g'})*(\sum b_{g_{2},g''}z_{g''})\\
 & =\sum_{g',g''}b_{g_{1},g'}b_{g_{2},g''}z_{g'}*z_{g''}\\
 & =\sum_{g',g''}v^{\lambda(g',g'')}b_{g_{1},g'}b_{g_{2},g''}\sum_{g}\alpha_{g',g''}^{g}z_{g}\\
 & =\sum_{g'\preceq g_{1},g''\preceq g_{2}}v^{\lambda(g',g'')}b_{g_{1},g'}b_{g_{2},g''}\sum_{g\preceq g'+g''}\alpha_{g',g''}^{g}\sum_{\eta\preceq g}c_{g,\eta}s_{\eta},
\end{align*}
where the sums are $\prec_{t_{0}}$-unitriangular and well-defined
in $\widehat{\LP(t_{0})}$. Now we have 
\begin{align*}
\supp\eta & \subset\supp g\subset(\supp g'\cup\supp g'')\subset(\supp g_{1}\cup\supp g_{2})
\end{align*}
and, respectively,
\begin{align*}
\suppDim\eta & \leq\suppDim g\leq(\suppDim g'+\suppDim g'')\leq(\suppDim g_{1}+\suppDim g_{2})
\end{align*}
such that $\suppDim\eta<\suppDim g_{1}+\suppDim g_{2}$ when $\eta\neq g_{1}+g_{2}$.
The desired claims follow.

\end{proof}

\subsection{Construction by freezing\label{subsec:Construction-by-freezing}}

Recall that we have defined $\clAlg^{\Theta}=\oplus_{g\in\Theta}\kk s_{g}$.
Define the injective $\kk$-linear map $\frz_{F}^{\cS}:\clAlg^{\Theta}\rightarrow\hLP(\frz_{F}t_{0})$
such that $\frz_{F}^{\cS}s_{g}:=\frz_{F,g}s_{g}$. Note that $\{\frz_{F,g}s_{g}|g\in\Theta\}$
is a $\Theta$-pointed basis of $\frz_{F}^{\cS}\clAlg^{\Theta}$.

\begin{Lem}

Assume that $\clAlg^{\Theta}$ is closed under multiplication, then
$\frz_{F}^{\cS}\clAlg^{\Theta}$ is closed under multiplication.

\end{Lem}

\begin{proof}

For any $s_{g_{1}},s_{g_{2}}$, we have a finite decomposition $v^{-\lambda(g_{1},g_{2})}s_{g_{1}}*s_{g_{2}}=\sum_{g}b_{g_{1},g_{2}}^{g}s_{g}$,
which must be $\prec_{t_{0}}$-unitriangular by Lemma \ref{lem:finit_sum_triangular}.
Applying $\frz_{F,g_{1}+g_{2}}$ on both sides, we get a $\prec_{\frz_{F}t_{0}}$-unitriangular
decomposition $v^{-\lambda(g_{1},g_{2})}\frz_{F,g_{1}}s_{g_{1}}*\frz_{F,g_{2}}s_{g_{2}}=\sum_{n}b_{g_{1},g_{2}}^{g}\frz_{g_{1}+g_{2}}s_{g}$.
By Lemma \ref{lem:projection_properties}, $\frz_{g_{1}+g_{2}}s_{g}$
is either $0$ or equals $\frz_{g}s_{g}$. Therefore, $\frz_{F}\clAlg^{\Theta}$
is a closed under multiplication with the basis $\frz_{F}\cS$.

\end{proof}

\begin{Lem}\label{lem:bases_with_transitions}

Assume that $\cZ=\{z_{g}|g\in\Theta\}$ is another basis of $\clAlg^{\Theta}=\oplus\kk s_{g}$.
Then $\{\frz_{F,g}z_{g}|g\in\Theta\}$ is another basis of $\frz_{F}^{\cS}\clAlg^{\Theta}$.
Equivalently, we have $\frz_{F}^{\cS}\clAlg^{\Theta}=\frz_{F}^{\cZ}\clAlg^{\Theta}$.

\end{Lem}

\begin{proof}

For any $s_{g}\in\cS$, we have a finite $\prec_{t_{0}}$-unitriangular
decomposition $s_{g}=\sum_{g'}b_{g,g'}z_{g'}$. By taking $\frz_{F,g}$,
we obtain

\begin{align*}
\frz_{F,g}s_{g} & =\sum_{g'}b_{g,g'}\frz_{F,g}z_{g'}
\end{align*}
where $\frz_{F,g}z_{g'}$ is either $0$ or equals to $\frz_{F,g'}z_{g'}$
by Lemma \ref{lem:projection_properties}. Therefore, we obtain that
$\frz_{F}^{\cS}\clAlg^{\Theta}:=\oplus_{g\in\Theta}\kk\frz_{F,g}s_{g}\subset\oplus_{g\in\Theta}\kk\frz_{F,g}z_{g}=:\frz_{F}^{\cZ}\clAlg^{\Theta}$.

By the same method, we obtain the inverse inclusion $\frz_{F}^{\cZ}\clAlg^{\Theta}\subset\frz_{F}^{\cS}\clAlg^{\Theta}$.
The claim follows.

\end{proof}

\begin{Rem}

Although we have $\frz_{F}^{\cS}\clAlg^{\Theta}=\frz_{F}^{\cZ}\clAlg^{\Theta}$
in Lemma \ref{lem:bases_with_transitions}, the $\kk$-linear maps
$\frz_{F}^{\cS}$ and $\frz_{F}^{\cZ}$ are not the same in general. 

\end{Rem}

Lemma \ref{lem:bases_with_transitions} implies the following result.

\begin{Thm}\label{thm:freezing_basis}

Assume that $\tB(t_{0})$ is of full rank, $\upClAlg(t_{0})$ has
a $\Z^{I}$-pointed basis $\cZ=\{z_{g}|g\in\Z^{I}\}$ such that $\deg z_{g}=g$,
and $\{\frz_{F,g}z_{g}|g\in\Z^{I}\}$ is a basis of $\upClAlg(\frz_{F}t_{0})$.
Then the following statement is true: if $\cS=\{s_{g}|g\in\Z^{I}\}$
is a basis for $\upClAlg(t_{0})$ such that $s_{g}$ are $g$-pointed,
then $\{\frz_{F,g}s_{g}|g\in\Z^{I}\}$ is a basis for $\upClAlg(\frz_{F}t_{0})$ 

\end{Thm}

\begin{Rem}

Assume that $\tB(t_{0})$ is of full rank and $t_{0}$ is injective-reachable.
Note that $\frz_{F}t_{0}$ is also injective-reachable by \cite{muller2015existence}.
Then Theorem \ref{thm:freeze_tropical_bases} implies that the assumption
in Theorem \ref{thm:freezing_basis} is satisfied in the following
cases:
\begin{enumerate}
\item $\kk=\Z$: we can take $\cZ$ to be the theta basis \cite{gross2018canonical}
or the generic basis \cite{qin2019bases}.
\item $\kk=\Z[v^{\pm}]$, and $B$ is skew-symmetric: we can take $\cZ$
to be the quantum theta basis \cite{davison2019strong}.
\item $\kk=\Z[v^{\pm}]$ and $\upClAlg(t_{0})$ possesses the common triangular
basis $\can$ in the sense of \cite{qin2017triangular}: we can take
$\cZ=\can$.
\end{enumerate}
\end{Rem}

\subsection{Construction by localization\label{subsec:Construction-by-localization}}

In this section, we give an alternative way to construct bases for
$\upClAlg(\frz_{F}t_{0})$. This construction is motivated by the
monoidal categorification of $\upClAlg(\frz_{F}t_{0})$ in \cite{qin2023analogs}.

As before, let $\Theta$ denote a subset of $\cone(t_{0})$ and $\cS=\{s_{g}|g\in\Theta\}$
a subset of $\hLP(t_{0})$ such that $s_{g}$ are $g$-pointed. Further
assume that $\Theta+f_{i}\subset\Theta$ for any $i\in I$.

\begin{Property}[Property (S)]\label{property:S}

We say $\cS$ has property (S),\footnote{We view $g+df_{k}$ as a shift of $g$ along the direction $f_{k}$,
and (S) stands for a shift.} if for any $s_{g}$ and any $k\in I_{\ufv}$, we have $k\notin\supp s_{g+df_{k}}$
when $d\in\N$ is large enough.

\end{Property}

\begin{Property}[Factorization Property]

Let $I'$ denote a subset of $I$. We say $\cS$ has the factorization
property in the direction $I'$ if, for any $s_{g}$ and any $j\in I'$,
we have $v^{-\lambda(f_{j},g)}x_{j}*s_{g}=s_{g+f_{j}}$.

\end{Property}

The three important families of bases in the introduction satisfy
the factorization property in the direction $I_{\fv}$.

\begin{Lem}\label{lem:shift_product_factorization}

Take any $k\in I_{\ufv}$ and choose $F=\{k\}$, $R=I_{\ufv}\backslash\{k\}$.
Assume that $\frz_{F}\cS$ has the factorization property with respect
to $k$: 
\begin{align*}
v^{-\lambda(f_{k},g)}x_{k}*\frz_{F}s_{g} & =\frz_{F}s_{g+f_{k}}.
\end{align*}
Then the following claims are true.

(1) We have $s_{g+df_{k}}=v^{-\lambda(df_{k},g)}x_{k}^{d}*\frz_{F}s_{g}=x_{k}^{d}\cdot\frz_{F}s_{g}$
when $k\notin\supp s_{g+df_{k}}$.

(2) Further assume that $\cS$ has local support under multiplication
(Definition \ref{def:local_support}) and $x_{k}^{d}=s_{df_{k}}$
for $d\in\N$. Then we have $s_{g+f_{k}}=v^{-\lambda(f_{k},g)}x_{k}*s_{g}=x_{k}\cdot s_{g}$
when $k\notin\supp s_{g}$.

\end{Lem}

\begin{proof}

(1) Because $k\notin\supp s_{g+df_{k}}$, we have $\frz_{F}s_{g+df_{k}}=s_{g+df_{k}}$.
We further have $\frz_{F}s_{g+df_{k}}=v^{-\lambda(df_{k},g)}x_{k}^{d}*\frz_{F}s_{g}$
by assumption. The claim follows.

(2) Because $k\notin\supp s_{g}$, we have $k\notin\supp s_{g+f_{k}}$
by considering the decomposition of $x_{k}*s_{g}$ into $\cS$ and
the local support assumption. Then we have both $\frz_{F}s_{g}=s_{g}$
and $\frz_{F}s_{g+f_{k}}=s_{g+f_{k}}$. Moreover, we have $v^{-\lambda(f_{k},g)}x_{k}*\frz_{F}s_{g}=\frz_{F}s_{g+f_{k}}$
by assumption. The claim follows.

\end{proof}

Denote $\Theta'=\{g\in\Theta|\supp s_{g}\cap F=\emptyset\}$. Define
$s'_{g}=s_{g}$ for any $g\in\Theta'$. Assume that $\cS$ has property
(S). For any $g\in\Theta\backslash\Theta',$ we define the $g$-pointed
localized function 
\begin{align}
s'_{g} & :=x^{-\sum_{j\in F}d_{j}f_{j}}\cdot s_{g+\sum_{j\in F}d_{j}f_{j}}\label{eq:localized_func}
\end{align}
for some choice of $d_{j}$ large enough as in Property (S).

\begin{Prop}\label{prop:induced_basis}

Assume that $\cS$ has Property (S) and $\frz_{F}\cS$ has the factorization
property in direction $F$. Then the following statements are true.

(1) We have $s_{g}'=\frz_{F,g}s_{g}$ for $g\in\Theta$. In particular,
$s'_{g}$ is independent of the choice of $d_{j}$.

(2) The set $\{s_{g}'|g\in\Theta\}$ is a basis of $\clAlg^{\Theta}(\frz_{F}t_{0})$.

\end{Prop}

\begin{proof}

(1) We have $s'_{g}=\frz_{F,g}s_{g}$ for any $g\in\Theta$ by Lemma
\ref{lem:shift_product_factorization}.

(2) The claim follows from (1).

\end{proof}

\begin{Thm}\label{thm:induced_theta_basis}

Assume that $\tB(t_{0})$ is of full rank, and $\cS=\{s_{g}|g\in\Z^{I}\}$
is a basis for $\upClAlg(t_{0})$ such that $s_{g}$ are $g$-pointed.
Assume that $\cS$ has Property (S) and $\frz_{F}\cS:=\{\frz_{F,g}s_{g}|g\in\Z^{I}\}$
has the factorization property in direction $F$. Let us take the
set of localized functions $\cS':=\{s'_{g}|g\in\Z^{I}\}$ as in (\ref{eq:localized_func}).
Then, in the cases of Theorem \ref{thm:freezing_basis}, $\cS'$ is
a basis of $\upClAlg(\frz_{F}t_{0})$ and it equals $\frz_{F}\cS$.

\end{Thm}

\begin{proof}

The claim follows from Proposition \ref{prop:induced_basis} for $\Theta=\Z^{I}$
and Theorem \ref{thm:freezing_basis}.

\end{proof}

\subsection{Applications to theta bases\label{sec:well-known-bases}}

As before, let there be given a (quantum) seed $t_{0}$ and a subset
$F\subset I_{\ufv}$. Then we have the seed $\frz_{F}t_{0}$ obtained
from $t_{0}$ by freezing the vertices $F$. Denote $R=I_{\ufv}\backslash F$.
We will see that we can obtain theta bases for $\upClAlg(\frz_{F}t_{0})$
from the theta bases for $\upClAlg(t_{0})$ via both the freezing
construction (Theorem \ref{thm:freezing_basis}) and the localization
construction (Theorem \ref{thm:induced_theta_basis}).

We work in the cases $\kk=\Z$, or $\kk=\Z[v^{\pm}]$ but $(b_{ij})_{i,j\in I_{\ufv}}$
is skew-symmetric. Assume that $t_{0}$ is injective-reachable, then
the $m$-pointed theta functions $\vartheta_{m}$, $m\in\Z^{I}$,
form a basis of $\upClAlg(t_{0})$ by \cite{gross2018canonical} \cite{davison2019strong}.
By Theorem \ref{thm:projection_theta_function}, their images $\frz_{F}\vartheta_{m}$
under the freezing operators are still theta functions, which are
known to satisfy the factorization property. Since $\frz_{F}t_{0}$
is still injective-reachable, we have the following result.

\begin{Thm}\label{thm:freeze_theta_basis}

The set $\{\frz_{F}\vartheta_{m}|m\in\Z^{I}\}$ is the theta basis
of $\upClAlg(\frz_{F}t_{0})$.

\end{Thm}

By the following Proposition \ref{prop:shifted_theta_func_supp},
the theta basis for $\upClAlg(t_{0})$ has Property (S). Therefore,
Theorem \ref{thm:induced_theta_basis} applies to it, since the theta
basis for $\upClAlg(\frz_{F}t_{0})$ also has the factorization property
in the frozen direction $F$.

\begin{Prop}\label{prop:shifted_theta_func_supp}

Choose a generic base point $Q$ in the interior of $\cC^{+}$. Let
there be given any $g\in\cone(t_{0})$ such that $\theta_{Q,g}=\theta_{g}$
is a Laurent polynomial. Then, for any $k\in I_{\ufv}$, we have $k\notin\supp\theta_{g+df_{k}}$
when $d\in\N$ is large enough.

\end{Prop}

\begin{proof}

Let us denote $\theta_{g}=\sum_{n\in\N^{I_{\ufv}}}c_{n}x^{g+p^{*}n}$.
We have the finite set $N_{g}:=\{n|c_{n}\neq0,n_{k}\neq0\}$.

For any $n\in N_{g}$, when $d$ is large enough, the following holds
for any $0\leq n'\leq n$ and any $0<n^{(j)}\leq n$ with the $k$-th
coordinate $n_{k}^{(j)}>0$: 
\begin{align}
\langle g+df_{k}+p^{*}n',n^{(j)}\rangle & >0.\label{eq:ineq_derivative}
\end{align}
We further choose $d$ large enough such that the above inequality
holds for any $n\in N_{g}$.

Denote $\theta_{g+df_{k}}=\sum c_{n}'x^{g+df_{k}+p^{*}n}$ and $N_{g+df_{k}}=\{n|c_{n}'\neq0,n_{k}\neq0\}$.
Note that we have a finite decomposition of positive Laurent polynomials
$x_{k}^{d}*\theta_{g}=v^{\alpha}\theta_{g+df_{k}}+\sum b_{g'}\theta_{g'}$
with $\alpha\in\Z$, $b_{g'}$ non-negative. It follows that $N_{g+df_{k}}\subset N_{g}$.

If $N_{g+df_{k}}=\emptyset$, we have verified the claim.

Otherwise, assume $N_{g+df_{k}}\neq\emptyset$. Then there exists
some $n\in N_{g+df_{k}}\subset N_{g}$ or ,equivalently, $\theta_{g+df_{k}}$
contains a Laurent monomial $x^{g+df_{k}+p^{*}n}$ with $n_{k}>0$.
Then this implies that there is a broken line $\gamma:(-\infty,0]\rightarrow\R^{I}$
which bends at walls $\frd^{(i)}\subset(n^{(i)})^{\bot}$ at times
$-\infty<t^{(1)}<\cdots<t^{(r)}<0$. Let $n^{(i)}>0$ denote the primitive
normal vectors of $\frd^{(i)}$. Moreover, for the linear segments
$L^{(i)}=\gamma(t^{(i)},t^{(i+1)}]$ (denote $t^{(0)}=-\infty$, $t^{(r+1)}=0$),
the derivative $\gamma'$ on $L^{(i)}$ is $-g^{(i)}:=-g-df_{k}-p^{*}(\sum_{s=1}^{i}c_{s}n^{(s)})$,
where $c_{s}\in\Z_{>0}$ are multiplicities, and $n=\sum_{s=1}^{r}c_{s}n^{(s)}$.
Since $n_{k}>0$, we must have $n_{k}^{(j)}>0$ for some $j$. Note
that $\gamma(t^{(j)})\in(n^{(j)})^{\perp}$. Define half spaces $H^{\pm}:=\{m\in\R^{I}|\pm\langle n^{(j)},m\rangle>0\}$.
By the inequality (\ref{eq:ineq_derivative}), we have $\langle-g^{(i)},n^{(j)}\rangle<0$
for any $i\geq j$. Equivalently, the derivatives of the segments
of $L^{(i)}$ belong to $H^{-}$ when $i\geq j$. So $\gamma(t^{(j)},0]$
is contained in $H^{-}$. This is impossible because we have $\gamma(0)=Q\in H^{+}$.

\end{proof}


\newcommand{\etalchar}[1]{$^{#1}$}
\def\cprime{$'$}
\providecommand{\bysame}{\leavevmode\hbox to3em{\hrulefill}\thinspace}
\providecommand{\MR}{\relax\ifhmode\unskip\space\fi MR }
\providecommand{\MRhref}[2]{%
  \href{http://www.ams.org/mathscinet-getitem?mr=#1}{#2}
}
\providecommand{\href}[2]{#2}


\end{document}